\pdfoutput=1
\RequirePackage{ifpdf}
\ifpdf 
\documentclass[pdftex]{sigma}
\else
\documentclass{sigma}
\fi

\usepackage[all]{xy}
\usepackage{bm}

\newcommand{\der}[2]{\frac{\partial #1}{\partial#2}}

\makeatletter
\renewcommand*\env@matrix[1][*\c@MaxMatrixCols c]{%
 \hskip -\arraycolsep
 \let\@ifnextchar\new@ifnextchar
 \array{#1}}

\numberwithin{equation}{section}

\newtheorem{Theorem}{Theorem}[section]
\newtheorem*{Theorem*}{Theorem}
\newtheorem*{Theorem-5.11}{Theorem~\ref{classificationaffine}}
\newtheorem*{Theorem-4.2}{Theorem~\ref{Lie-Cartan}}
\newtheorem{Corollary}[Theorem]{Corollary}
\newtheorem{Lemma}[Theorem]{Lemma}
\newtheorem{Proposition}[Theorem]{Proposition}
 { \theoremstyle{definition}
\newtheorem{Definition}[Theorem]{Definition}

\newtheorem{Exemple}[Theorem]{Example}
\newtheorem{Remark}[Theorem]{Remark} }

\begin{document}
\allowdisplaybreaks

\newcommand{\arXivNumber}{2210.02251}

\renewcommand{\PaperNumber}{052}

\FirstPageHeading

\ShortArticleName{Single-Valued Killing Fields of a Meromorphic Affine Connection and Classification}

\ArticleName{Single-Valued Killing Fields of a Meromorphic\\ Affine Connection and Classification}

\Author{Alexis GARCIA}

\AuthorNameForHeading{A.~Garcia}

\Address{Laboratoire de Math\'ematiques Blaise Pascal, Universit\'e Clermont Auvergne, France}
\Email{\href{mailto:alexis.garcia@uca.fr}{alexis.garcia@uca.fr}}

\ArticleDates{Received November 09, 2022, in final form July 17, 2023; Published online July 27, 2023}

\Abstract{We give a geometric condition on a meromorphic affine connection for its Killing vector fields to be single valued. More precisely, this condition relies on the pole of the connection and its geodesics, and defines a subcategory. To this end, we use the equivalence between these objects and meromorphic affine Cartan geometries. The proof of the previous result is then a consequence of a more general result linking the distinguished curves of meromorphic Cartan geometries, their poles and their infinitesimal automorphisms, which is the main purpose of the paper. This enables to extend the classification result from [Biswas~I., Dumitrescu~S., McKay~B., \textit{\'Epijournal G\'eom. Alg\'ebrique} \textbf{3} (2019), 19, 10~pages, arXiv:1804.08949] to the subcategory of meromorphic affine connection described before.}

\Keywords{meromorphic affine connections; Killing vector fields; infinitesimal automorphisms; Cartan geometries}

\Classification{53A15; 53B05; 53B15; 53C05; 53C15; 57N16}

\section{Introduction}

\subsection{Geometric structures and Cartan geometries}
A class of smooth geometric structures on real manifolds, or holomorphic geometric structures on complex manifolds (see \cite[p.~65]{Gromov} for a modern definition) is obtained as infinitesimal versions of a model geometry. As an example in the smooth category, the notion of Riemannian metric is obtained as an infinitesimal version of Euclidean geometry, and affine connections as infinitesimal versions of affine geometry. These two classes of geometric structures were intensively studied, in particular by Riemann who initiated with Gauss the \textit{Riemannian geometry}.

In the two examples above, we remark that the group of global \textit{automorphisms} of the model geometry, i.e., global transformations preserving the characteristics of this geometry, acts transitively on the base space, namely $\mathbb{R}^n$. We say that the geometric structure corresponding to the model geometry is \textit{homogeneous}. This property was later proposed by Klein to give 	a definition of a geometry, in his famous program aimed at classifying all the geometries. A \textit{Klein geometry} is a couple $(G,P)$ formed by a Lie group $G$, seen as the group of global automorphisms of the geometry, and a closed Lie subgroup $P$ seen as the subgroup of isotropy at a fixed point of the space $G/P$.

Geometric structures underlying a Klein geometry are of diverse kinds. A general fact is that the model space $G/P$ is endowed with a \textit{$Q$-structure} where $Q$ is a linear subgroup naturally associated with $P$. The geometric structures obtained in this way are of order one, but some Klein geometries define higher order geometric structures. For example, for the affine Klein geometry, where $G$ is the affine group of $\mathbb{R}^n$ and $P$ the linear subgroup, the group $G$ is exactly the group of global automorphisms for the canonical flat affine connection on $\mathbb{R}^n$. In general, the geometric structure underlying $(G,P)$ is defined using the $P$-principal bundle $G\longrightarrow G/P$ and the Maurer--Cartan form $\omega_G$ of $G$.

In a series of papers, in particular \cite{Cartan2}, Cartan described affine connections as infinitesimal versions of the affine Klein geometry, and proposed to generalize this principle to any Klein geometry to obtain a \textit{Cartan geometry}. The formalism used nowadays came, in affine case, from the works of Ehresmann, who gave a purely geometric definition of an affine connection in terms of a principal bundle, a \textit{principal connection} \cite[p.~154]{Erhesmanh}, and a soldering form to give a geometric meaning to the principal bundle. There is also an equivariant definition which was proposed by Atiyah in \cite[p.~188]{Atiyah}, and which is useful to extend some results in the meromorphic setting (see Section~\ref{sectionequivalence}). The definition of a Cartan geometry on $M$ modelled on $(G,P)$ that will be adopted in this article is the following: a couple $(E,\omega)$ formed by a $P$-principal bundle over $M$ and a $\mathfrak{g}$-valued equivariant one-form on $E$ mimicking the infinitesimal properties of the Maurer--Cartan form of $G$.

 In this way, the principle constructing a geometric structure on $G/P$ from a Klein geometry $(G,P)$ can be generalized to Cartan geometries, except that it produces non-homogeneous geometric structures in general: the global \textit{automorphisms} of the Cartan geometry do not act transitively on the base manifold.

 \subsection{Classification of quasi-homogeneous geometric structures}
 A natural question then, going back to the work of Riemann, Hopf and Killing for Riemannian metrics (see, for example, \cite{Wolf}), is to classify \textit{locally homogenous} geometric structures, i.e., for which the \textit{infinitesimal automorphisms} span the tangent space of the base manifold at any point. As an example, it is well known that any locally homogeneous and complete Riemannian metric on a simply connected manifold is homogeneous.

 The above question is more relevant in the holomorphic category, for two principal reasons. First, the existence of a holomorphic geometric structure on a complex compact manifold gives some restrictions on its geometry or its topology. Secondly, local homogeneity is sometimes deduced from the complex geometry of the base manifold, at least on an open dense subset. These two reasons are well illustrated by the holomorphic version of Riemannian metrics, i.e., holomorphic fields of nondegenerate bilinear forms on the tangent spaces of a complex manifold (see, for example, \cite[definition on p.~1663]{DumitrescuMetriques} and \cite[definition on p.~210]{LeBrun}). Indeed, on a general complex manifold, such an object gives a trivialisation of some power of the canonical bundle. On a~compact complex surface, the curvature of such a object is a constant function, implying local homogeneity.

 In dimension two, we can also mention the work by Inoue, Kobayashi and Ochiai in \cite{InoueKobayashi}. Using the vanishing of the first two Chern classes of a complex compact surface in presence of a holomorphic affine connection and the Enriques--Kodaira classification, they gave a complete classification of such objects. In particular, any compact complex surface admitting a~holomorphic affine connection admits a flat holomorphic affine connection, which is thus locally homogeneous.

In \cite{McKay}, McKay showed that the existence of an arbitrary holomorphic Cartan geometry on a complex K\"ahler manifold imposes relations on its Chern classes. In a common paper with Dumitrescu \cite{BD}, they proved that a simply connected compact complex manifold, with \textit{algebraic dimension} zero (i.e., whose meromorphic functions are the constants), does not bear any holomorphic affine connection. This results is stated for a vast class of models, called of \textit{algebraic type} (see \cite[p.~1]{Dumitrescu1}).

 Dumitrescu gave in \cite{Dumitrescu1} a result in arbitrary dimension which implies that on compact complex manifolds with only constant meromorphic functions, any holomorphic Cartan geometry must be \textit{quasi-homogeneous}, i.e., locally homogeneous on an open dense subset. This was used in \cite{BDM} by Biswas and the two previous authors to improve the above result.

 \begin{Theorem}\label{theoremBDM} Compact complex manifolds, whose meromorphic functions are constants, bearing a holomorphic Cartan geometry of algebraic
type $($see above$)$ have infinite fundamental group.
 \end{Theorem}
Two important facts are used in the proof. First, it is proved that any germ of infinitesimal automorphism of the Cartan geometry is the germ of a global infinitesimal automorphism: this is a generalization of a result by Nomizu \cite{Nomizu} for analytic Riemannian metrics, and follows from the fact that the former objects form a local system on $M$. Hence, using the result by Dumitrescu mentioned above, there are $n=\dim(M)$ independent germs of infinitesimal automorphisms at some point of $M$, extending as a family of $n$ global holomorphic vector fields which are infinitesimal automorphisms. Next, it is proved that there exists a commuting family of vector fields with the previous property. Hence, there exists an action of a complex abelian Lie group~$L$ with an open dense orbit in $M$. The conclusion follows from detailed study of the geometry of such manifolds $M$, which implies that the Cartan geometry is flat.

\subsection{Results of the paper}
In this paper, we consider the meromorphic generalization of the holomorphic geometric structures, in particular the \textit{meromorphic affine connections}. In the meromorphic category, the two above facts no longer stand: infinitesimal automorphisms could be multivalued (see \cite[Example~3.8]{BlazquezCasale}). Moreover, meromorphic single valued infinitesimal automorphism may not have a well defined flow at some point of the pole.

 We give a sufficient condition on some meromorphic Cartan geometries to recover the first fact. Let us explain briefly the condition. In the remainder of this paper, a \textit{pair} is a couple~$(M,D)$ where $M$ is a complex manifold and $D$ a divisor on $M$. The holomorphic $P$-principal bundle~$E$ of a meromorphic Cartan geometry $(E,\omega_0)$ on a pair $(M,D)$ (see Definition~\ref{Cartan}), modelled on a~complex Klein geometry $(G,P)$, comes equipped with holomorphic foliations (possibly singular) $\mathcal{T}_A$ whose leaves are the \text{$A$-distinguished curves} (see definition in Section~\ref{sectionspiral}, or \cite[Definition~4.16]{Sharpe} in the non-singular setting). In the case of the affine model, it is natural to consider these leaves because their projections $\Sigma$ on the base manifold $M$ are exactly the spirals of the corresponding meromorphic affine connection (see \cite[p.~344]{Sharpe}).

 Moreover, there is a well-known result in Riemannian geometry stating that any Killing vector field $X$ for a Riemannian metric $g$ is a Jacobi field, i.e., for any geodesic $\Sigma$ of $\nabla$, the scalar product $g(\gamma'(t),X(\gamma(t)))$ is constant along $\gamma$. We can translate the two objects in terms of the Cartan geometry $(E,\omega)$ corresponding to the Levi-Civita connection of $g$. The proof that any Killing vector field is single-valued can then be recovered from the fact that $(E,\omega)$ is torsionfree and from the structure of the Lie algebra $\mathfrak{g}$ of the complex Euclidean group~$G$.

In the meromorphic setting, and for an arbitrary model, we prove the result below. It involves the following objects: the sheaf $\mathcal{V}$ of holomorphic functions on $E$ with values in $\mathfrak{g}$, for any $A\in \mathfrak{g}$ (resp.\ $\mathcal{V}\bigl(*\tilde{D}\bigr)$ for the meromorphic functions with poles at $\tilde{D}$, the inverse image of $D$), the subsheaf $\mathcal{V}_A$ of functions with values in $\mathbb{C}A$, $\pi_{\mathcal{V}/\mathcal{V}_A}$ the corresponding projection, and $\mathcal{K}$ the subsheaf of $\mathcal{V}$ formed by the images of the local infinitesimal automorphisms of the meromorphic Cartan geometry $(E,\omega_0)$.

\begin{Theorem-4.2}\label{Lie-Cartan-Int}
Let $(G,P)$ be a complex Klein geometry, and $(M,D)$ be a pair with $\dim(M)=\dim(G/P)$. Let $(E,\omega_0)$ be a meromorphic $(G,P)$-Cartan geometry on $(M,D)$, $\mathcal{V}=\mathcal{O}_E\otimes \mathfrak{g}$ and~$\mathcal{K}$ the sheaf defined as in \eqref{K}. Let $x_0 \in D$ belonging to the smooth part of an unique irreducible component $D_\alpha$, and $\tilde{D}_\alpha = p^{-1}(D_\alpha)$.

Suppose that there exist $A\in \mathfrak{g}\setminus \mathfrak{p}$ and a $A$-distinguished curve $\tilde{\Sigma}$ for $(E,\omega_0)$ such that $\tilde{\Sigma}\cap \tilde{D}_\alpha= \{e_0\}$ for some point $e_0$ in the fiber of $x_0$.

Denote by $\mathcal{T}_A$ the distinguished foliation of $E$ \eqref{definitionTA}, as well as $\mathcal{V}_A=\Phi_{\omega_0}(\mathcal{T}_A)$ $($see \eqref{Phi}$)$. Then there exists a neighborhood $U$ of $x_0$ with the following properties:
\begin{itemize}\itemsep=0pt
\item[$(1)$] Let $s$ be a section of $\mathcal{K}$ on an open subset $V\subset p^{-1}(U\setminus D)$. Then the class $[s]_{\mathcal{V}/\mathcal{V}_A}$ of $s$ in $TE/\mathcal{T}_A$ extends as a $($single-valued$)$ section of $TE/\mathcal{T}_A$ on $p^{-1}(U\setminus D)$.
 \item[$(2)$] The image under $\Phi_{\omega_0}$ $($see \eqref{Phi}$)$ of the section defined in $(1)$ is the restriction of a section of $\mathcal{V}/\mathcal{V}_A(* D)$ over $p^{-1}(U)$.
\item[$(3)$] Suppose moreover that $(E,\omega_0)$ is holomorphic branched on $(M,D)$. Then the above section lies in $\mathcal{V}/\mathcal{V}_A\big(p^{-1}(U)\big)$. \end{itemize}
\end{Theorem-4.2}

 The meromorphic Cartan geometries satisfying this condition of the above theorem for a~generic point $x_0\in D$ on the pole are said to be \textit{strongly spiral} (see Definitions~\ref{spiral} and~\ref{strongtaucondition}). This theorem implies that Killing fields of meromorphic affine connections which are strongly spiral are single valued and meromorphic. If we restrict ourselves to the subcategory of \textit{branched holomorphic affine connections}, i.e., those arising from \textit{branched holomorphic Cartan geometries} (see~\cite{BD}), we obtain that the Killing vector fields may be seen as holomorphic sections of submodule $\mathcal{E}\subset TM(* D)$ satisfying $TM\subset \mathcal{E}$. Using this fact and some results in complex geometry, we obtained the following partial generalization of Theorem~\ref{theoremBDM}.

\begin{Theorem-5.11}\label{classificationaffine-Int} Let $M$ be a compact complex manifold with finite fundamental group, and whose meromorphic functions are constants. Then $M$ does not bear any spiral branched holomorphic affine connection.
\end{Theorem-5.11}

\subsection{Plan of the paper}
The plan of the paper is as follow. In Section~\ref{section1}, we recall the dictionary between locally free modules of finite rank and vector bundles, the corresponding meromorphic sections, and recall the definition of Atiyah's exact sequence associated with a principal bundle. In Section~\ref{section2}, we introduce meromorphic Cartan geometries and the holomorphic vector bundles naturally associated to these objects. In Section~\ref{killparab}, we prove Theorem~\ref{Lie-Cartan} and deduce from this result sufficient conditions for two classes of regular meromorphic \textit{parabolic} geometries (see Section~\ref{sectionparabolique}) to have single-valued infinitesimal automorphisms. We say in this case that the geometry satisfies the \textit{extension property} (see Definition~\ref{extensionproperty}). In Section~\ref{section4}, we prove the equivalence between meromorphic affine Cartan connections and meromorphic affine connections, and we introduce the \textit{spiral connections} (see Definition~\ref{tauconnection}). We moreover restricts ourselves to a subcategory of meromorphic affine connection, namely the \textit{branched} affine connections (see Definition~\ref{branchedaffine}). This is because in this subcategory, the \textit{spiral connections} are induced by a meromorphic affine Cartan geometry satisfying the sufficient condition described before. We then use the extension property to prove the classification result Theorem~\ref{classificationaffine}. Finally, in the last section, we discuss the genericity of the spiral connections, and illustrate~Theorem~\ref{classificationaffine} by examples in any dimension.

 \section{Preliminaries and notations}\label{section1}
This preliminary section is devoted to recall the notion of meromorphic connections on a locally free module, and the meromorphic version of Atiyah's exact sequence associated with a principal bundle. This enables us to extends the equivalence proved by C. Ehresmann in \cite{Erhesmanh} between principal connections and linear connections to the meromorphic setting.

\subsection{Locally free modules and meromorphic connections}
Let $(M,D)$ be a \textit{pair}, i.e., a complex manifold $M$ equipped with a divisor $D$, we denote by $\mathcal{O}_M$ the sheaf of holomorphic functions on $M$ and $\mathcal{M}_M$ the sheaf of meromorphic functions on it. In order to write statements about meromorphic objects with poles at $D$, we may use the sheaf $\mathcal{O}_M(* D)$ of meromorphic functions with poles supported on the irreducible components~$D_\alpha$ of~$D=\sum_{\alpha} n_\alpha D_\alpha$ (see \cite[p.~17]{Sabbah}). Let $\mathcal{L}$ be a coherent $\mathcal{O}_M$-module. Then we can consider the sheaf
 \begin{gather*}
 \mathcal{L}(* D) = \mathcal{O}_M(* D)\otimes_{\mathcal{O}_M} \mathcal{L} \end{gather*}
 of meromorphic sections of $\mathcal{L}$ with poles supported at the irreducible components of $D$. The \textit{order} $\operatorname{ord}^\mathcal{L}_D(s)$ at $D$ of a section $s$ of $\mathcal{L}(* D)$ defined on an open subset $U\subset M$ is the greatest integer $d\in \mathbb{Z}$ such that $s$ is also a section of $\mathcal{L}(-d D)$ on $U$.
\begin{Definition}\label{connection} A \textit{meromorphic connection} on $(M,D)$ is a couple $(\mathcal{V},\nabla)$ where $\mathcal{V}$ is a locally free $\mathcal{O}_M$-module of finite rank, and $\nabla$ is a morphism of $\mathbb{C}$-sheaves from $\mathcal{V}(* D)$ to $\Omega^1_M \otimes \mathcal{V}(* D)$ satisfying the Leibniz identity $\nabla(fs) = {\rm d}(f) s + f\nabla(s)$ for any $s\in \mathcal{V}(U)$ and $f\in \mathcal{O}_M(* D)(U)$ ($U$ is an open subset of $M$).
\end{Definition}
If $(\mathcal{L},\nabla)$ and $(\mathcal{L}',\nabla)$ are two meromorphic connections such that $\mathcal{L}= \bigoplus_{i=1}^{r} \mathcal{O}_M s_i$, $\mathcal{L}'= \bigoplus_{i=1}^{r} \mathcal{O}_M t_i$ and $t_i = \sum_{j=1}^{r} q_{ji} s_j$ for a meromorphic matrix $Q$ on $M$, then the matrices $A$ and $A'$ respectively associated to the basis $(s_i)_{i=1,\dots,r}$ and $(t_i)_{i=1,\dots,r}$ are linked by the gauge-transformation formula (see \cite[p.~29]{Sabbah})
\begin{gather*}
 A' = Q^{-1}{\rm d}Q + Q^{-1}AQ,
\end{gather*}
where ${\rm d}$ stands for the de Rham derivative.

A \textit{meromorphic affine connection} on $(M,D)$ is a meromorphic connection $\nabla$ on $TM$ with poles supported at $D$. The \textit{torsion} of a meromorphic affine connection $\nabla$ on $(M,D)$ is the meromorphic section $T_\nabla$ of $\Omega^1_M \otimes \operatorname{End}(TM)$ defined by
\begin{gather}\label{torsionaffine}
 T_\nabla(X)(Y) = \nabla_X(Y) - \nabla_Y(X) - [X,Y]_{TM}. \end{gather}

Now, let describe the categories associated with the above objects. Let $M$, $M'$ be two complex manifolds and $f\colon M\longrightarrow M'$ be a holomorphic map. We denote by \begin{gather*}
 f^{-1}\mathcal{O}_{M'} \end{gather*}
 the pullback in the sheaf theoretic sense. This is a sheaf of algebras over the constant sheaf $\underline{\mathbb{C}}_M$. Then $\mathcal{O}_M$ is naturally a $f^{-1}\mathcal{O}_{M'}$-algebra through the $\underline{\mathbb{C}}_M$-algebras morphism
\begin{align*}
f_\# \colon\ f^{-1}\mathcal{O}_{M'} & \longrightarrow \mathcal{O}_M,\nonumber \\
 s & \mapsto s\circ f.
\end{align*}
 Hence, any $f^{-1}\mathcal{O}_{M'}$-module defines a $\mathcal{O}_M$-module obtained by tensorizing with $\mathcal{O}_M$ in the category of $f^{-1}\mathcal{O}_{M'}$-modules.

Recall that there is a well-known equivalence between the category of holomorphic vector bundles $V$ of rank $r\geq 1$ over complex manifolds and locally free $\mathcal{O}_M$-modules of the same rank obtained by considering the sheaf of local holomorphic sections of $V$ (see \cite[Proposition~4.1]{Sabbah}). Let $\hat{\Psi}\colon V_1 \longrightarrow V_2$ be an isomorphism of vector bundles $V_1$, $V_2$ over two complex mani\-folds~$M$,~$M'$. Denote by $\mathcal{E}_1$, $\mathcal{E}_2$ the associated sheaves of sections, and $f\colon M \longrightarrow M'$ the isomorphism of complex manifolds covered by $\hat{\Psi}$. Then the arrow obtained as the image of $\hat{\Psi}$ under this equivalence of categories is the isomorphism $\Phi$ of $\mathcal{O}_M$-modules defined by
 \begin{align}
 \Phi \colon\ \mathcal{E}_1 & \longrightarrow \mathcal{O}_M\otimes_{f^{-1}\mathcal{O}_{M'}} f^{-1}\mathcal{E}_2,\nonumber\\
 s& \longrightarrow \hat{\Psi} \circ s \circ f^{-1},\label{isovector}
 \end{align}
 where $s$ stands for a local section of $\mathcal{E}_1$. \begin{Definition}\label{isovect}A couple $(f,\Phi)$ as above will be called a \textit{isomorphism of vector bundles} between~$\mathcal{E}_1$ and $\mathcal{E}_2$. More generally, we define a \textit{isomorphism of meromorphic bundles} by replacing the sheaves of holomorphic sections by the corresponding of meromorphic sections with poles at a divisor $D$ of $M$ and $D'$ of $M'$. \end{Definition}

The following definition of the pullback of a meromorphic connection is a particular case of the construction of inverse images for $D$-modules (see \cite[p.~21]{Hotta}). Let $\nabla'$ be a meromorphic connection on the $\mathcal{O}_{M'}$-module $\mathcal{E}'$ with poles at $D'$ and $f\colon M\longrightarrow M'$ be a morphism of complex manifolds. The dual of the differential ${\rm d}f\colon TM \longrightarrow \mathcal{O}_M \otimes f^{-1}TM'$ is a morphism of $\mathcal{O}_M$-modules $({\rm d}f)^* \colon \mathcal{O}_M \otimes f^{-1} \Omega^1_{M'} \longrightarrow \Omega^1_M$.

Now, let ${\rm d}_M$ and ${\rm d}_{M'}$ denote respectively the de Rham differentials on $\mathcal{O}_M$ and $\mathcal{O}_{M'}$. Since ${\rm d}_M \circ f_\# = ({\rm d}f)^* \circ f^{-1} {\rm d}_{M'}$ by definition of $f_\#$, we remark that the left composition
\[\overline{f^{-1}\nabla'}\colon\ f^{-1}\mathcal{E}' \longrightarrow \Omega^1_M \otimes \mathcal{O}_M \otimes f^{-1}(\mathcal{E}'(* D'))\]
 of the sheaf theoretic pullback $f^{-1}\nabla'$ with $({\rm d}f)^*\otimes {\rm Id}$ satisfies
 \begin{equation}\label{nablabar} \overline{f^{-1}\nabla'}(f_\#(g')s) = {\rm d}_M(f_\#(g')) \otimes s + g' \otimes \overline{f^{-1}\nabla'}(s) \end{equation}
 for any section $s$ of $f^{-1}\mathcal{E}'$ and $g'$ of $f^{-1}\mathcal{O}_{M'}$. Thus, it extends as a morphism of $\mathbb{C}_{M}$-modules
 \begin{equation}\label{pullbackeq} f^\star \nabla' \colon \ \mathcal{O}_M\otimes f^{-1}\mathcal{E}' \longrightarrow \Omega^1_M \otimes \big(\mathcal{O}_M\otimes f^{-1}\mathcal{E}'(* D')\big) \end{equation}
 satisfying the Leibniz rule on $M$.

\begin{Definition}\label{pullback}Let $f$ and $(\mathcal{E}',\nabla')$ as above. The pullback of $(\mathcal{E}',\nabla')$ is the pair $(\mathcal{E},\nabla)$, where $\mathcal{E}=\mathcal{O}_M \otimes f^{-1}\mathcal{E}'$ and $\nabla$ is the morphism defined as in \eqref{pullbackeq}. \end{Definition}

When $\mathcal{O}_M\otimes f^{-1}\mathcal{O}_{M'}(D') = \mathcal{O}_M(D)$ for some divisor $D$ on $M$, the pullback $(\mathcal{E},f^\star \nabla')$ is a~meromorphic connection on $(M,D)$. This the case, for example, when $f$ is a submersion, but we may find many counterexamples as the following example.

\begin{Exemple}\label{exemplegeodesicisotrope} Let $M'=\mathbb{C}^2$ and $D'=\{z_1=0\}$. Let $\mathcal{E}'=\mathcal{O}_{M'}^{\oplus 2}$ and $\nabla'$ be the meromorphic connection on $\mathcal{E}'$, with poles at $D'$, whose matrix in the canonical basis is $\frac{{\rm d}z_2}{z_1} \otimes {\rm Id}$.

Let $\gamma\colon \mathbb{C} \longrightarrow M'$ be the holomorphic curve defined by $\gamma(t)=(0,t)$. Then $\mathcal{O}_\mathbb{C} \otimes \gamma^{-1}\mathcal{E}'= \mathcal{O}_\mathbb{C}^{\oplus 2}$ and $\gamma^\star \nabla'$ is the null morphism. \end{Exemple}
\begin{proof} We have $\mathcal{O}_\mathbb{C} \otimes \gamma^{-1} \mathcal{O}_{M'}(D') = \underline{\{0\}}_{\mathbb{C}}$ because $\gamma(\mathbb{C})\subset D'$. Now, by definition of $\nabla'$, the morphism $\overline{\gamma^{-1}\nabla'}$ defined before \eqref{nablabar} maps $\gamma^{-1}\mathcal{E}'$ in $\Omega_\mathbb{C}^1 \otimes \mathcal{O}_\mathbb{C} \otimes \gamma^{-1}\big(\frac{1}{z_1} \mathcal{E}'\big)$. By the first remark, this is the trivial module. \end{proof}

We finish with the Riemann--Hilbert correspondence. A \textit{flat} meromorphic connection $\nabla$ on~$\mathcal{E}$ with poles at $D$ is a meromorphic connection such that the subsheaf of \textit{horizontal sections} $\ker(\nabla)$ on $M\setminus D$ defined by
\begin{gather*}
 \forall U\subset M\setminus D,\quad \ker(\nabla)(U) = \{ s\in \mathcal{E}(U)\, \text{s.t.}\, \nabla(s)=0\} \end{gather*}
 is a \textit{local system} (see \cite[Chapter~I, Theorem~2.17]{Deligne}).

We recall \cite[Chapter~I,  Theorem~2.17]{Deligne} that there is an equivalence of categories between the category of local systems of rank $r$ on $M\setminus D$ with arrows being the isomorphisms, and the category of representations $\rho \colon \pi_1(M\setminus D,x) \longrightarrow K$ (for any $x\in M\setminus D$, and $K$ is a $\mathbb{C}$-vector space of dimension $r$) with arrows being the isomorphisms of representations. Once a point $x\in M\setminus D$ is chosen, this equivalence is obtained by associating to any local system $\mathcal{K}$, the \textit{monodromy map} $\operatorname{Mon}^x(\mathcal{K})\colon \pi_1(M\setminus D,x) \longrightarrow \operatorname{Aut}(\mathcal{K}_x)$ (for the definition, see \cite[Chapter~I, Section~1.2]{Deligne}).

\subsection{Atiyah sequence of the frame bundle}\label{Atiyahframe}

We recall and extend the Atiyah exact sequence (see \cite[Theorem 1]{Atiyah}) in the meromorphic setting.

The \textit{frame bundle} of a locally free $\mathcal{O}_M$-module $\mathcal{E}$ of rank $r$ is the holomorphic ${\rm GL}_r(\mathbb{C})$-principal bundle $E\overset{p}{\longrightarrow} M$ whose fiber at $x\in M$ is the set of isomorphisms $\mathbb{C}^r \simeq \mathcal{E}(x)$. Here $\mathcal{E}(x)=\mathcal{E}_x /\mathfrak{m}_x$ stands for the fiber of $\mathcal{E}$ at $x$.

We recall that for any complex Lie group $P$ and a holomorphic $P$-principal bundle $E\overset{p}{\longrightarrow} M$, there is a notion of $P$-\textit{linearization} for a $\mathcal{O}_E$-module $\mathcal{V}$ (or of a $P$-\textit{equivariant} sheaf, see \cite[Definition 9.10.3]{Hotta}). This is a family $(\phi_b)_{b\in P}$ of isomorphisms $\phi_b\colon \mathcal{V} \simeq r_b^* \mathcal{V}$ (where $r_b$ is the right action of $P$) with nice properties. A $\mathcal{O}_E$-module equipped with a $P$-linearization is said to be $P$-\textit{equivariant}. In this context, there is an equivalence between the $P$-equivariant locally free $\mathcal{O}_E$-modules and the locally free $\mathcal{O}_M$-modules, and between the $P$-equivariant morphisms and the morphisms between the corresponding $\mathcal{O}_M$-modules (see \cite[Lemma, p.~3]{Lunts}). For any representation $\rho\colon P \longrightarrow {\rm GL}(\mathbb{V})$, and any holomorphic $P$-principal bundle $E\overset{p}{\longrightarrow} M$, we denote by\looseness=-1
\begin{gather}\label{representationmodule} E(\mathbb{V})
\end{gather}
 the $\mathcal{O}_M$-module associated with the $\mathcal{O}_E$-module $\mathcal{O}_E \otimes \mathbb{V}$, where the $P$-linearization $(\phi_b)_{b\in P}$ is given by $\phi_b = r_b^* \otimes \rho\big(b^{-1}\big)$. We call it the \textit{representation module} associated with $E$ and $\mathbb{V}$. For any isomorphism $\Psi\colon E\longrightarrow E'$ of holomorphic $P$-principal bundles covering $\varphi\colon M\longrightarrow M'$, the \textit{representation isomorphism of associated vector bundles} corresponding to $\Psi$ is the isomorphism
\begin{gather}\label{representationmorphism} \Psi(\mathbb{V})\colon\ E(\mathbb{V}) \longrightarrow \varphi^* E'(\mathbb{V})
 \end{gather}
 associated to the $P$-equivariant isomorphism $\Psi^* \otimes {\rm Id}_{\mathbb{V}}$ of trivial $\mathcal{O}_E$-modules.
\begin{Definition}\label{meromorphicisoprincipal}
Let $\mathbb{V}$ be a representation of a complex Lie group $P$. Let $E\overset{p_1}{\longrightarrow} M$ and $E'\overset{p_1}{\longrightarrow} M'$ be two holomorphic $P$-principal bundles, and $D$, $D'$ be respectively two divisors of~$M$ and~$M'$.
\begin{itemize}\itemsep=0pt
\item[$(1)$] An isomorphism $\Psi E\vert_{M\setminus D} \longrightarrow E'\vert_{M'\setminus D'}$ of holomorphic $P$-principal bundles is \textit{$\mathbb{V}$-mero\-mor\-phic} between $(M,D)$ and $(M',D')$ iff the representation isomorphism $\Phi=\Psi(\mathbb{V})$ restricts to an isomorphism $\Phi\colon E(\mathbb{V})(* D) \longrightarrow \varphi^* E'(\mathbb{V})(* D')$ (see \eqref{representationmorphism}).
\item[$(2)$] A \textit{$\mathbb{V}$-meromorphic section} of a holomorphic $P$-principal bundle $E\overset{p}{\longrightarrow} M$ on $U$ with pole at $D$ is a holomorphic section $\sigma\colon U\setminus D \longrightarrow E$ such that the corresponding trivialisation $\psi_\sigma$ of $E(\mathbb{V})\vert_{U\setminus D}$ induces an isomorphism of meromorphic bundles between $E(\mathbb{V})$ and $\mathcal{O}_U \otimes \mathbb{V}$.
 \end{itemize} \end{Definition}
In particular, mapping holomorphic ${\rm GL}_r(\mathbb{C})$-principal bundles $E$ over $M$ to the associated representation modules $E(\mathbb{C}^r)$ gives an equivalence of categories. A pseudo-inverse is given by mapping a locally free $\mathcal{O}_M$-module $\mathcal{E}$ of rank $r$ to its frame bundle $E$.

Consider $\mathfrak{p}={\rm Lie}(P)$ which is the adjoint representation of $P$. Let $\operatorname{At}(E)$ be the $\mathcal{O}_M$-module associated with the $P$-equivariant locally free $\mathcal{O}_E$-module $TE$ equipped with the $P$-linearization induced by the infinitesimal action of $P$ on $E$: it is called the \textit{Atiyah bundle} of $E$, and fits into the short exact sequence (see \cite[Theorem 1]{Atiyah})
\begin{gather}\label{Atiyah} \xymatrix{ 0 \ar[r] & E(\mathfrak{p}) \ar[r]^{\iota} & \operatorname{At}(E) \ar[r]^q & TM \ar[r]& 0,} \end{gather}
 where $\iota$ is the morphism associated with the $P$-equivariant morphism which to any $A \in \mathcal{O}_E \otimes \mathfrak{p}$ associates the corresponding fundamental vector field on $E$, and $q$ is the one associated with the $P$-equivariant morphism ${\rm d}p\colon TE \longrightarrow p^*TM$.

The previous equivalence implies that $P$-equivariant meromorphic one-forms on $E$, with poles at $\tilde{D}=p^{-1}(D)$, and values in $\mathbb{V}$ are in bijection with morphisms $\beta\colon \operatorname{At}(E)(* D) \longrightarrow E(\mathbb{V})(* D)$, or equivalently with sections of $\operatorname{At}(E) \otimes E(\mathbb{V})(* D)$. This correspondence restricts to a bijective correspondence between
 \begin{itemize}\itemsep=0pt
 \item[$\bullet$] The set of morphisms $\beta$ as above vanishing on the image of $\iota$ in \eqref{Atiyah}, equivalently sections of $\Omega^1_M(* D) \otimes E(\mathbb{V})$.
 \item[$\bullet$] The set of meromorphic one-forms $\tilde{\omega}$ on $\big(E,\tilde{D}\big)$ with values in $\mathbb{V}$ vanishing on $\ker({\rm d}p)$. \end{itemize}

\subsection{Meromorphic principal connections and meromorphic connections}\label{sectionprincipale}

We now extend the well-known equivalence between linear connections on a vector bundle and principal connections on its frame bundle to the equivalence between meromorphic connections on a locally free $\mathcal{O}_M$-module $\mathcal{E}$ and meromorphic principal connections on its frame bundle~$E$. It straightforwardly restricts as an equivalence between meromorphic connections preserving a holomorphic reduction $\overline{E}\subset E$ to a subgroup $P\subset {\rm GL}_r(\mathbb{C})$, and meromorphic $P$-principal connections on $E$. In the non-singular setting, this was first proved by C.~Ehresmann \cite{Erhesmanh} using the formalism of horizontal lifts for paths, and reformulated in an equivariant way by M.~Atiyah~\cite{Atiyah}. We adopt the point of view of M.~Atiyah in order to extend the result to the meromorphic category.

 The starting point is that for $P={\rm GL}_r(\mathbb{C})$, there is a canonical isomorphism \cite[Proposition~9]{Atiyah}
 \begin{gather*}
 E(\mathfrak{p}) = \operatorname{End}(\mathcal{E}).
 \end{gather*}
There is a bijection between the set of meromorphic connections $\nabla$ on $\mathcal{E}$ and the one of $\mathcal{O}_M$-linear splittings $\delta\colon \mathcal{E}(* D) \longrightarrow J^1(\mathcal{E})(* D)$ of the exact sequence of $\mathbb{C}$-sheaves
 \begin{gather}\label{jetsequence} \xymatrix{ 0 \ar[r] & \Omega^1_M(* D) \otimes \mathcal{E} \ar[r] & J^1(\mathcal{E})(* D) \ar[r] & \mathcal{E}(* D) \ar[r] & 0, }
 \end{gather}
 where $J^1(\mathcal{E})$ is the jet-module of $\mathcal{E}$ (see \cite[p.~193]{Atiyah}). Let $\sigma\colon U \longrightarrow E$ be a holomorphic section of the holomorphic frame bundle. This corresponds to a basis $(s_1,\dots,s_r)$ of $\mathcal{E}\vert_U$, and we denote in the following lines by $d$ the pullback of the de Rham differential through the corresponding isomorphism $\mathcal{E}\vert_U \simeq \mathcal{O}_U^{\oplus r}$. The former equivalence is given by $\nabla = {\rm d} - \delta$. Indeed, this clearly defines a meromorphic connection, and if $\nabla$ is a meromorphic connection on $\mathcal{E}\vert_U$, then $\delta_1 = {\rm d}-\nabla$ is a morphism of $\mathcal{O}_U$-modules from $\mathcal{E}\vert_U(* D)$ to $\Omega^1_U(* D) \otimes \mathcal{E}\vert_U$, and we obtain a splitting $\delta = ({\rm Id}_{\mathcal{E}\vert_U}, \delta_1)$ of \eqref{jetsequence}.

\begin{Definition}\label{principal} A \textit{meromorphic principal connection} on a holomorphic ${\rm GL}_r(\mathbb{C})$-principal bundle $E\overset{p}{\longrightarrow}M$ with poles at $\tilde{D}=p^{-1}(D)$ \big(shortly on $\big(E,\tilde{D}\big)$\big) is a meromorphic one-form $\tilde{\omega}$ on~$E$ with values in $\mathfrak{p}$, which is $P$-equivariant and such that $\tilde{\omega}$ coincides with the Maurer--Cartan form of $P$ when restricted to any fiber $p^{-1}(x) \subset E$. \end{Definition}
 Using the correspondence for equivariant one-forms as in Section~\ref{Atiyahframe}, a meromorphic $P$-principal connection on $\big(E,\tilde{D}\big)$ is equivalent to a morphism $\beta\colon \operatorname{At}(E)(* D) \longrightarrow E(\mathfrak{p})(* D)$ such that $\iota \circ \beta = {\rm Id}_{\operatorname{At}(E)}$, where $\iota$ is defined in \eqref{Atiyah}. Its kernel defines a splitting
\begin{gather}\label{tau}
 \tau\colon\ TM(* D) \longrightarrow \operatorname{At}(E)(* D)
 \end{gather}
 of \eqref{Atiyah}, which uniquely determines $\beta$. The following lemma straightforwardly follows from the equivalence described before between equivariant morphisms of modules over principal bundles and morphisms between the corresponding modules over the base manifolds.

\begin{Lemma}\label{lemmetau} Let $(M,D)$ and $(M',D')$ be two pairs of same dimension. Let $\tilde{\Psi}\colon E \longrightarrow E'$ be an isomorphism of holomorphic $P$-principal bundles over $M$ and $M'$ covering a morphism of pairs $\varphi\colon M \longrightarrow M'$ $($i.e., $\varphi(D)=D'$$)$. Let $\tilde{\omega}_2$ be a meromorphic principal connection on $\big(E,\tilde{D}_1\big)$ where $\tilde{D}_1$ is the preimage of $D$,
respectively $\tilde{\omega}_1=\tilde{\Psi}^\star \tilde{\omega}_2$, and $\tau_1$ $($resp.\ $\tau_2)$ is the splitting associated to $\tilde{\omega}_1$ $($resp.\ $\tilde{\omega}_2)$ as in~\eqref{tau}.
Then the diagram below is commutative
\begin{gather*}
 \xymatrix{ TM \ar[r]^-{\tau_1} \ar[d]_{{\rm d}\varphi} & \operatorname{At}(E)(* D) \ar[d]^{p_* {\rm d}\tilde{\Psi}} \\ \varphi^* TM' \ar[r]_-{\varphi^*\tau_2} & \varphi^* \operatorname{At}(E')(* D), }
 \end{gather*}
 where $p$ is the foot map of $E$.
\end{Lemma}

Denote by $\tilde{{\rm d}}$ the usual de Rham differential on $\mathcal{O}_E\bigl(*\tilde{D}\bigr) \otimes \mathbb{V}$. Since the $P$-linearization $\big(\phi^{\mathbb{V}}_b\big)_{b\in P}$ preserves the subsheaf of constant functions with values in $\mathbb{V}$ on $E$, the pushforward~$p_* \tilde{{\rm d}}$ restricts to $p_* \tilde{{\rm d}}\colon \mathcal{E} \longrightarrow p_* \Omega_E^1 \otimes \mathcal{E}$. This defines a meromorphic connection $\nabla$ on $\mathcal{E}$ by
 \begin{equation}\label{covariant} \nabla = \tau \lrcorner p_* \tilde{{\rm d}}, \end{equation}
 where $\tilde{{\rm d}}$ is defined above and $\lrcorner$ stands for the contraction by a vector field.
\begin{Proposition}\label{correspondenceprincipal} Mapping a meromorphic principal connection $(E,\tilde{\omega})$ over $(M,D)$ to the meromorphic connection $(\mathcal{E},\nabla)$ on $(M,D)$ defined by \eqref{covariant} induces an equivalence of categories between
\begin{itemize}\itemsep=0pt
 \item The category of principal meromorphic $($resp.\ holomorphic$)$ connections over pairs, where the arrows are the $\mathbb{C}^r$-meromorphic isomorphisms $($see Definition~{\rm \ref{meromorphicisoprincipal})} of principal bundles between pairs preserving the principal connections $($resp.\ isomorphisms of holomorphic principal bundles preserving the principal connections$)$.
 \item The category of meromorphic $($resp.\ holomorphic$)$ connections on $(M,D)$ with isomorphisms of meromorphic bundle $($resp.\ holomorphic vector bundles, see Definition~{\rm \ref{isovect})} preserving connections $($in the sense of Definition~{\rm \ref{pullback})}. \end{itemize} \end{Proposition}

\begin{proof} Let us first prove that this map induces a functor. Let $\tilde{\Psi}\colon E \longrightarrow E'$ be an isomorphism of meromorphic principal connections on ${\rm GL}_r(\mathbb{C})$-principal bundles between $(E_1,\tilde{\omega}_1)$ and $(E_2,\tilde{\omega}_1)$, over $(M,D)$. Let $(\mathcal{E}_1,\nabla_1)$ and $(\mathcal{E}_2,\nabla_2)$ be the meromorphic connections obtained by $\mathcal{E}_i=E_i(\mathbb{C}^r)$ (see \eqref{representationmodule}) and $\nabla_i$ defined as in \eqref{covariant}. Let $(\varphi,\Phi)$ be the isomorphism of vector bundles (at the level of sheaves, see Definition~\ref{isovect}) where $\varphi$ is the isomorphism of complex manifolds covered by $\tilde{\Psi}$ and where $\Phi:=\tilde{\Psi}(\mathbb{C}^r)$ (see \eqref{representationmorphism}).

Fix any open subset $U\subset M$ and a basis $(s_i)_{i=1,\dots,r}$ of $\mathcal{E}_1\vert_U$ and denote by $(\varphi^* t_i)_{i=1,\dots,r}$ its image through $\Phi$. Denote by $(\tilde{s}_i)_{i=1,\dots,r}$ and $\big(\tilde{t}_{i=1,\dots,r}\big)$ respectively the corresponding equivariant functions on $p_1^{-1}(U)$ and $p_2^{-1}(\varphi(U))$. Thus $\tilde{t}_i = \tilde{s}_i \circ \tilde{\Psi}$ by definition of $\Phi$. By definition of~$\Phi^{-1} \varphi^\star \nabla_2$, we can compute
\begin{gather*}
\Phi^{-1} \varphi^\star \nabla_2(s_i) = \big({\rm Id}_{\Omega^1_{M}} \otimes \Phi^{-1}\big)[{\rm d}\varphi \lrcorner (\varphi^* \nabla_2 (\varphi^* t_i))]. \end{gather*}
Using the definition of $\nabla_1$ and $\nabla_2$, and Lemma~\ref{lemmetau}, we get
\begin{align*}
 \Phi^{-1} \varphi^\star \nabla_2(s_i)&= \big({\rm Id}_{\Omega^1_{M}} \otimes \Phi^{-1}\big)\big[ (\varphi^*\tau_2\circ {\rm d}\varphi) \lrcorner \varphi^* p_{2*}\tilde{{\rm d}}_2\big(\tilde{t}_i\big)\big]\\
 &= \tau_1 \lrcorner \big(p_{1*}\tilde{{\rm d}}_1\tilde{t}_i \circ \tilde{\Psi}\big) \\
 &=\nabla_1(s_i), \end{align*}
 where we denoted by $\tilde{{\rm d}}_1$ and $\tilde{{\rm d}}_2$ the usual de Rham differentials on $\mathcal{O}_{E}\otimes \mathbb{C}^r$ and $\mathcal{O}_{E'}\otimes \mathbb{C}^r$. Hence we can map $\tilde{\Psi}$ to the vector bundle isomorphism $(\varphi,\Phi)$ which preserves the linear meromorphic connections $\nabla_1$ and $\nabla_2$.

 Now, we construct the pseudo-inverse. Let $(\mathcal{E},\nabla)$ be a meromorphic connection over a~pair $(M,D)$. Denote by $E$ its frames bundle. Let $x\in M$ and $U$ be a neighborhood equipped with a holomorphic section $\sigma\colon U \longrightarrow E$. Denote by $(s_1,\dots,s_r)$ the corresponding basis of $\mathcal{E}\vert_U$. The section $\sigma$ induces a splitting $TE\vert_{p^{-1}(U)} = p^* TU \oplus \ker({\rm d}p)$ which is $P$-equivariant, hence a~splitting
 \begin{equation}\label{splitAt} \operatorname{At}(E)\vert_U = TU \oplus E(\mathfrak{p})\vert_U.
 \end{equation}
 We denote by $\tau_0$ the splitting of the exact sequence \eqref{Atiyah} restricted to $U$ induced by \eqref{splitAt}, and by ${\rm d}$ the pullback of the de Rham differential through the trivialization associated with $(s_i)_{i=1,\dots,r}$. Let $\delta = \nabla-{\rm d}$, which vanishes on the image of $E(\mathfrak{p})$ through $\iota$ (see \eqref{Atiyah}). Its kernel thus define a morphism $\Theta\colon TU \longrightarrow \operatorname{At}(E)\vert_U(* D)$, and we obtain a splitting
 \begin{gather*}
 \tau= \tau_0 +\Theta
 \end{gather*}
 of \eqref{Atiyah} over $U$. From the remarks above, this is equivalent to a meromorphic principal connection $\tilde{\omega}_U$ on $p^{-1}(U)$ with poles at $\tilde{D}\cap p^{-1}(U)$.

 Now, let $U$, $U'$ be two open subset and $(s_i)_{i=1,\dots,r}$ and $(s'_i)_{i=1,\dots,r}$ be two basis of $\mathcal{E}\vert_U$ and $\mathcal{E}\vert_{U'}$ corresponding to holomorphic sections $\sigma$, $\sigma'$ of $E$ on $U$ and $U'$. Let ${\rm d}$ and ${\rm d}'$ be the corresponding de Rham differentials, then
 \begin{gather*}
 {\rm d}-{\rm d}'(s'_i) = {\rm d}(s'_i) = {\rm d}\bigg(\sum_{j=1}^{r} b^{-1}_{ji} s_j\bigg) = \sum_{j=1}^{r}\big(b{\rm d}_0\big(b^{-1}\big)\big)_{ji} s'_j,
 \end{gather*}
 where $b$ is the meromorphic function on $U\cap U'$ with values in $P$ such that $\sigma' = \sigma \cdot b$, and ${\rm d}_0$ is the usual de Rham differential on $\mathfrak{p}$-valued functions. Denote by $\tau$ and $\tau'$ constructed as before. Thus
\[
 \tau'-\tau= [(\sigma,b^\star \omega_P)].\]
 Thus $\nabla'-\nabla = {\rm d}'-{\rm d} + \tau'-\tau =0$ and the corresponding meromorphic principal connections~$\tilde{\omega}$ and $\tilde{\omega}'$ coincide over $p^{-1}(U\cap U')$. We obtain a global meromorphic principal connection $\tilde{\omega}$ on~$\big(E,\tilde{D}\big)$ inducing $\nabla$ as in \eqref{covariant}.

 If $(\varphi,\Phi_0)$ is an isomorphism of vector bundles preserving the meromorphic connections~$\nabla_1$,~$\nabla_2$, then from Section~\ref{Atiyahframe} it induces an isomorphism $\tilde{\Psi}$ of holomorphic principal bundles between~$E$ and $E'$. Since the action of $P$ on $\mathbb{C}^r$ is free, by definition of $\nabla_1$ and $\nabla_2$ we get that $\varphi^*\tau_2 = \tau_1 \circ {\rm d}\varphi$. By Lemma~\ref{lemmetau}, we obtain $\tilde{\Psi}^\star \tilde{\omega}_2 = \tilde{\omega}_1$.
\end{proof}

\section{Spiral meromorphic Cartan geometries}\label{section2}
In this section, we fix a pair $(M,D)$ where $M$ is of complex dimension $n$. We recall the definition of meromorphic Cartan geometries, and the subcategory of branched holomorphic Cartan geometries introduced by Biswas and Dumitrescu in \cite{BD}.
We introduce the \textit{meromorphic extensions} (see Definition~\ref{meromorphicextension}). This will enable us to extend the equivalence between affine Cartan geometries and affine connections (see Proposition~\ref{equivalenceaffineextension}) and to use arguments of complex geometry in the proof of Theorem~\ref{classificationaffine}.

Following \cite{BlazquezCasale} and \cite{Cap4}, we recall the description of infinitesimal automorphisms as sections for a meromorphic connection either on a trivial module over the principal bundle of the geometry, or on the corresponding module over the base manifold.

We finally introduce the subcategory of \textit{strongly spiral} meromorphic Cartan geometries: in the next section, we will see that their infinitesimal automorphisms are single valued, in a sense that will be defined.

\subsection{Meromorphic and holomorphic branched Cartan geometries}
First, we have recall the definition the models for Cartan geometries (we refer to \cite[Chapter~4]{Sharpe}).
\begin{Definition}\label{Klein} A \textit{complex Klein geometry} of dimension $n\geq 1$ is a~couple $(G,P)$, where $G$ is a~complex Lie group, and $P$ is a closed complex Lie subgroup with $\dim(G)-\dim(P)=n$. \end{Definition}
Let $(G,P)$ be as in Definition~\ref{Klein} and let $P'=\ker(\underline{{\rm Ad}})$, where $\underline{{\rm Ad}}\colon P \longrightarrow {\rm GL}(\mathfrak{g}/\mathfrak{p})$ is the representation induced by the adjoint representation. Then any choice of a basis for $\mathfrak{g}/\mathfrak{p}$ identifies $Q=P/P'$ with a linear complex subgroup, and $TG/P$ with the module $G(\mathfrak{g}/\mathfrak{p})$ associated to the $P$-principal bundle $E$ and the representation $\mathfrak{g}/\mathfrak{p}$. Thus, the complex manifold $G/P$ comes equipped with a holomorphic reduction $G\times_{P} Q$ of its holomorphic frame bundle $R^1(G/P)$, i.e., a holomorphic $Q$-structure: namely $G/P'$.

 This is in fact only due to the presence of a $\mathfrak{g}$-valued holomorphic 1-form with special properties on the total space of the holomorphic $P$-principal bundle $G\longrightarrow G/P$, namely the \textit{Maurer--Cartan form} $\omega_G$ of $G$. We can consider curved versions of theses objects for which the above fact is still true replacing $G$ by a suitable holomorphic $P$-principal bundle (see next subsection). Authorizing the one-form to have poles on the $P$-principal bundle, we obtain their meromorphic analogues.

\begin{Definition}\label{Cartan} Let $(G,P)$ be a complex Klein geometry with $\dim(G/P)=n$ and $(M,D)$ be a~pair. A \textit{meromorphic $(G,P)$-Cartan geometry} is a couple $(E,\omega_0)$ where $E\overset{p}{\to}M$ is a~holomorphic $P$-principal bundle, and $\omega_0$ is a $\mathfrak{g}$-valued meromorphic one-form on $E$, with poles on $\tilde{D}=p^{-1}(D)$, such that
\begin{itemize}\itemsep=0pt
\item[(i)] For any $x\in M\setminus D$, $\iota_x^\star \omega_0$ coincides with the Maurer--Cartan form $\omega_{E,x}$ (see above).
\item [(ii)]$\omega_0$ is $P$-equivariant.
\item[(iii)] For any $e\in E\setminus \tilde{D}$, $\omega_0(e)$ is an isomorphism between $T_e E$ and $\mathfrak{g}$.
\end{itemize}
\end{Definition}
These objects form a category.

\begin{Definition}\label{isoCartan} Let $(G,P)$ be a complex Klein geometry and $(M,D)$, $(M',D')$ be two pairs with $\dim(M)=\dim(M')=\dim(G/P)$. Let $(E,\omega_0)$ and $(E',\omega'_0)$ be respectively two meromorphic $(G,P)$-Cartan geometries on $(M,D)$ and~$(M',D')$. An \textit{isomorphism} between $(E,\omega_0)$ and~$(E',\omega'_0)$ is an isomorphism of $\mathfrak{g}$-meromorphic $P$-principal fiber bundles $\Psi\colon E\setminus \tilde{D} \overset{\sim}{\to} E' \setminus \tilde{D}'$ (see~Definition~\ref{meromorphicisoprincipal}) such that $\Psi^\star \omega'_0 = \omega_0$. \end{Definition}
The following object is central in the study of Cartan geometries.

\begin{Definition}\label{curvature} Let $(E,\omega_0)$ be a meromorphic $(G,P)$-Cartan geometry on $(M,D)$. Its \textit{curvature function} (or \textit{curvature}) is the meromorphic function $k_{\omega_0}$ on $E$ with values in $\mathbb{W}=\bigl(\bigwedge^2 \mathfrak{g}^*\bigr)\otimes \mathfrak{g}$ and defined by
\[k_{\omega_0} = {\rm d}\omega_0 \circ \big(\omega_0^{-1} \wedge \omega_0^{-1}\big) + [\,,\,]_{\mathfrak{g}}, \]
where $[\,,\,]_\mathfrak{g}$ is the Lie-bracket of $\mathfrak{g}$ identified with an element of $\mathbb{W}$. \end{Definition}

Fix a Klein geometry $(G,P)$ and choose a basis $(\mathfrak{e}_i)_{i=1,\dots,N}$ of $\mathfrak{g}$, with $(\mathfrak{e}_i)_{i=1,\dots,n}$ spanning a~subspace $\mathfrak{g}_-$ complementary to $\mathfrak{p}$. Denote by $(\mathfrak{e}_i^*)_{i=1,\dots,N}$ the dual basis of $\mathfrak{g}^*$.
\begin{Definition}\label{constantes}Let $(E,\omega_0)$ be a meromorphic $(G,P)$-Cartan geometry on $(M,D)$. The meromorphic functions
\[\gamma_{i,j}^k= \mathfrak{e}_k^* \circ k_{\omega_0}(\mathfrak{e}_i,\mathfrak{e}_j) \]
 are called the \textit{structure coefficients} (or \textit{structure functions}) of $(E,\omega_0)$. \end{Definition}

A natural subcategory of the meromorphic $(G,P)$-Cartan geometries on pairs was introduced by Biswas and Dumitrescu (see~\cite{BD}).

\begin{Definition}\label{branchedCartan} A \textit{branched holomorphic $(G,P)$-Cartan geometry} on a pair $(M,D)$ is a meromorphic $(G,P)$-Cartan geometry $(E,\omega_0)$ on $(M,D)$ such that $\omega_0$ extends as a holomorphic one-form on~$E$.
\end{Definition}

An important feature of these objects for the classification is the existence of a holomorphic connection on the adjoint vector bundle (see \cite[p.~7]{BD}). We recall below its construction which will be useful.

Let $(E,\omega_0)$ be a branched holomorphic $(G,P)$-Cartan geometry on $(M,D)$, and $E_G=E\times_{P} G$ the extension of the holomorphic $P$-principal bundle $E$ to the group $G$. By definition, $E_G$ is the quotient of the product $E\times G$ by the action of $G$ given by $(e,g)\cdot h = \big(e\cdot h,h^{-1}g\big)$. Consider the $G$-equivariant holomorphic one-form $\overline{\omega}$ on $E\times G$ with values in $\mathfrak{g}$ given by
 \begin{gather*}
 \overline{\omega}= \operatorname{Ad}(\pi_2)\circ \pi_1^\star \omega_0 + \pi_2^\star \omega_G, \end{gather*}
 where $\pi_1$, $\pi_2$ are the projections on each factor and $\omega_G$ is the Maurer--Cartan form of $G$. It is straightforward to verify that for any $A\in \mathfrak{g}_0$, $\overline{\omega}\big(\frac{\rm d}{{\rm d}t}\vert_{t=0}(e,h)\cdot \exp_G(tA)\big) =0$, i.e., the vectors tangent to the fibers of
 \begin{gather*}
 \pi_G\colon\ E\times G \longrightarrow E_G
 \end{gather*}
 are in the kernel of $\overline{\omega}$. Moreover, $\overline{\omega}$ is invariant under the action of $G$ on $E\times G$. Thus, $\overline{\omega}$ induces a holomorphic one-form on $E_G$.
\begin{Definition}\label{tractorconnection} The holomorphic $G$-principal connection $\tilde{\omega}$ on $E_G$ induced by $\overline{\omega}$ as above is the \textit{tractor-connection} of $(E,\omega_0)$. We denote by $\nabla^{\omega_0}$ the corresponding holomorphic connection on $E_G(\mathfrak{g})=E(\mathfrak{g})$ (see Proposition~\ref{correspondenceprincipal}). \end{Definition}

The pullback $p^* E(\mathfrak{g})$ is the trivial module $\mathcal{V}=\mathcal{O}_E\otimes \mathfrak{g}$.

\begin{Lemma}\label{curvaturetractor} The pullback $p^\star \nabla^{\omega_0}$ is ${\rm d} - \operatorname{ad}(\omega_0)$ where ${\rm d}$ is the de Rham differential on the trivial module $\mathcal{V}$, $\operatorname{ad}(\omega_0)$ is the section of $\Omega^1_E \otimes \operatorname{End}(\mathfrak{g}) = \Omega^1_E \otimes \operatorname{End}(\mathcal{V})$ defined by
\[X\lrcorner \operatorname{ad}(\omega_0)(s)= [\omega_0(X),s]_\mathfrak{g}\] for any holomorphic vector field $X$ of $E$ and section $s$ of $\mathcal{V}$. In particular, its curvature is $R_{p^\star \nabla^{\omega_0}} = \operatorname{ad}({\rm d}\omega_0) + \operatorname{ad}(\omega_0\wedge \omega_0)$. \end{Lemma}
\begin{proof}Since the $\omega_0$-constant vector fields on $E$ span $T_eE$ at any $e\in E\setminus \tilde{D}$, we can choose $\tilde{A}=\omega_0^{-1}(A)$ for $A\in \mathfrak{g}$ as a holomorphic vector field on $E\setminus \tilde{D}$. Let $s$ be any section of $\mathcal{V}(U)$, $U\subset E\setminus \tilde{D}$ an open subset. By definition of $\nabla^{\omega_0}$ and the remarks preceding Definition~\ref{tractorconnection}, we have
\[\tilde{A}\lrcorner p^\star \nabla^{\omega_0}(s) =\big(\overline{A}-\hat{A}\big)\lrcorner {\rm d}(\tilde{s}),\]
where $\tilde{s}$ is the unique $G$-equivariant section of $\mathcal{O}_{E\times G}\otimes \mathfrak{g}$ which coincides with $s$ in restriction to $E\subset E\times G$, $\overline{A}$ is the unique $G$-invariant meromorphic vector field whose restriction to $E$ coincides with $\tilde{A}$, and $\hat{A}$ is the holomorphic vector field tangent to the fibers of $E\times G \overset{\pi_1}{\longrightarrow} E$ such that $\pi_2^\star \omega_G\big(\hat{A}\big)=A$. Indeed, $\overline{A}-\hat{A}$ is the unique vector field which belongs to $\ker(\overline{\omega})$ and projects to $\tilde{A}$ via $\pi_1\colon E\times G\longrightarrow E$. Now, $\overline{A} \lrcorner {\rm d}(\tilde{s})$ coincides with $\tilde{A}\lrcorner {\rm d}(\tilde{s})$ in restriction to $E$, while $\hat{A}\lrcorner {\rm d}(\tilde{s}) = [A,s]_\mathfrak{g}$ because $\tilde{s}$ is $G$-equivariant. The first formula follows. For the curvature, it corresponds to the classical computation of the curvature in a trivialisation of a~vector bundle.
\end{proof}

\subsection{Meromorphic extension of the tangent sheaf}
We now introduce an object induced by any meromorphic Cartan geometry, which plays the same role as the tangent bundle of the base manifold in the regular case. It is a particular case of the following objects.

\begin{Definition}\label{meromorphicextension} Let $(M,D)$ be a pair.
\begin{itemize}\itemsep=0pt
\item[(1)] A meromorphic extension of $(M,D)$ is a couple $(\phi_0,\mathcal{E})$ where $\mathcal{E}$ is a locally free $\mathcal{O}_M$-module and $\phi_0\colon TM(* D) \longrightarrow \mathcal{E}(* D)$ is an isomorphism of $\mathcal{O}_M$-modules.
\item[(2)] A holomorphic extension of $(M,D)$ is a meromorphic extension $(\phi_0,\mathcal{E})$ such that
    \[{\phi_0(TM)\subset \mathcal{E}}.\]
\item[(3)] The category $\mathcal{F}$ \big(resp.\ $\mathcal{F}^0$\big) of meromorphic extensions (resp.\ holomorphic extensions) over pairs is defined as follow. An arrow between two meromorphic extensions $(\phi_0,\mathcal{E})$ and~$(\phi'_0,\mathcal{E}')$ over $(M,D)$ and $(M',D')$ is an isomorphism $(\varphi,\Phi)$ of meromorphic bundles (resp.\ of vector bundles, see Definition~\ref{isovect}) between $\mathcal{E}$ and $\mathcal{E}'$ such that the following diagram commutes:
    \begin{equation*} \xymatrix{ TM \ar[r]^{\phi_0} \ar[d]_{{\rm d}\varphi} & \mathcal{E}(* D) \ar[d]^{\Phi} \\ \varphi^* TM' \ar[r]_-{\varphi^* \phi'_0} & \varphi^* \mathcal{E}'(* D). } \end{equation*}
 \item[(4)] The category obtained by restricting to meromorphic extensions of $(M,D)$ and to isomorphisms of meromorphic bundles of the form $({\rm Id}_M,\Phi)$ is denoted by $\mathcal{F}_{M,D}$ \big(resp.\ $\mathcal{F}^0_{M,D}$\big).
\end{itemize}
\end{Definition}

Meromorphic extensions on $(M,D)$ are thus canonically isomorphic to submodules of maximal rank of the sheaf of tangent vector fields with poles at $D$. The restriction of the corresponding frame bundle to $M\setminus D$ can thus be canonically identified with the frame bundle of $M\setminus D$. This gives the following alternative description.

\begin{Definition}\label{soldermeromorphe}Let $(M,D)$ be a pair.
\begin{itemize}\itemsep=0pt
\item[(1)] Let $E\overset{p}{\longrightarrow} M$ be a holomorphic $P$-principal bundle and $\tilde{D}=p^{-1}(D)$. A \textit{meromorphic solderform} on $\big(E,\tilde{D}\big)$ is a $P$-equivariant $\mathbb{C}^n$-valued meromorphic 1-form $\theta_0$ on $E$, with poles supported at $\tilde{D}$, vanishing on $\ker({\rm d}p)$, and such that $\theta_0(e)$ is surjective for any $e\in E\setminus \tilde{D}$. A couple $(E,\theta_0)$ is called a meromorphic solder form over $(M,D)$
\item[(2)] An arrow between two meromorphic solderforms $(E,\theta_0)$ and $(E',\theta'_0)$ over $(M,D)$ and $(M',D')$ is an isomorphism of holomorphic $P$-principal bundles $\tilde{\Psi}\colon E \longrightarrow E'$ such that $\theta_0 = \tilde{\Psi}^\star \theta'_0$. This defines the category $\mathcal{D}$ of meromorphic solderforms over pairs.
\end{itemize}
\end{Definition}

\begin{Proposition}\label{solderequivalence} The map which to any meromorphic solder form $(E,\theta_0)$ over $(M,D)$ $($see Definition~{\rm \ref{soldermeromorphe})} associates the meromorphic extension $(\phi_0,\mathcal{E})$ where $\phi_0\colon TM(* D) \overset{\sim}{\longrightarrow} \mathcal{E}(* D)$ is the isomorphism which corresponds to $\theta_0$ $($see remarks above$)$, extends to an equivalence of categories $m\colon \mathcal{D} \longrightarrow \mathcal{E}$.
 \end{Proposition}\begin{proof}
 If $\tilde{\Psi}\colon E \longrightarrow E'$ is an arrow between two objects $(E,\theta_0)$ and $(E',\theta'_0)$ of the category of solderforms over $(M,D)$, we define $m(\tilde{\Psi})=\Phi$ as the image of $\tilde{\Psi}$ through the equivalence of categories described in Section~\ref{Atiyahframe}. Consider the images $(\phi_0,\mathcal{E})$ and $(\phi'_0,\mathcal{E})$ of $(E,\theta_0)$ and $(E',\theta'_0)$. Since $\theta'_0 = \tilde{\Psi}^\star \theta_0$, by definition, $\Phi \circ \phi_0 = \phi'_0$ so $m$ is an essentially surjective functor. Since it is the restriction of the equivalence of categories described in Section~\ref{Atiyahframe}, it is an equivalence of categories.
\end{proof}

Now let $(E,\omega_0)$ be any meromorphic $(G,P)$-Cartan geometry on $(M,D)$. Then the meromorphic one-form $\pi_{\mathfrak{g}/\mathfrak{p}}\circ \omega_0$ obtained by projecting $\omega_0$ on $\mathfrak{g}/\mathfrak{p}$ is $P$-equivariant for the quotient adjoint action on $\mathfrak{g}/\mathfrak{p}$, and pointwise surjective on $E\setminus p^{-1}(D)$. Moreover, its kernel contains $\ker({\rm d}p)$. By the Section~\ref{Atiyahframe}, it thus corresponds to a morphism of $\mathcal{O}_M$-modules
\begin{gather*}
 \phi_0 \colon\ TM(* D) \longrightarrow \mathcal{E}(* D),
 \end{gather*}
 where we set $\mathcal{E}=E(\mathfrak{g}/\mathfrak{p})$. By construction, $\phi_0$ is an isomorphism of meromorphic bundles and $(\mathcal{E},\phi_0)$ is thus a meromorphic extension on $(M,D)$.

\begin{Definition}\label{fonct} The meromorphic extension $(\mathcal{E},\phi_0)$ obtained as above is the \textit{meromorphic extension induced by} $(E,\omega_0)$. We denote by $f$ the map from the set of meromorphic $(G,P)$-Cartan geometries on pairs to the set of meromorphic extensions which maps $(E,\omega_0)$ to its induced meromorphic extension $(\mathcal{E},\phi_0)$. This extends as a functor $f$ between the corresponding categories. \end{Definition} 

\subsection{Infinitesimal automorphisms as horizontal sections}
We recall the definition of the infinitesimal automorphisms of a Cartan geometry and their description as horizontal sections for a connection (see \cite[Lemma 1.5.12]{Cap}) and extend straightforwardly these facts to the meromorphic setting.

\begin{Definition}\label{kill} Let $(E,\omega_0)$ be a meromorphic $(G,P)$-Cartan geometry on $(M,D)$. An \textit{infinitesimal automorphism} of $(E,\omega_0)$ is a holomorphic vector field $\overline{X}$ on an open subset $U\subset M\setminus D$, lifting to a vector field $X$ on $p^{-1}(U)$ such that $\phi_X^{t \star} \omega_0 = \omega_0$. We write $\mathfrak{k}\mathfrak{i}\mathfrak{l}\mathfrak{l}^{\rm loc}_{M,\omega_0}$ for the subsheaf of~$TM\setminus D$ whose sections are the local infinitesimal automorphisms, and $\mathfrak{k}\mathfrak{i}\mathfrak{l}\mathfrak{l}^{\rm loc}_{E,\omega_0}$ for the subsheaf of $TE\setminus \tilde{D}$ whose sections are the lifts of sections of $\mathfrak{k}\mathfrak{i}\mathfrak{l}\mathfrak{l}^{\rm loc}_{M,\omega_0}$. \end{Definition}

In order to study the sections of $\mathfrak{k}\mathfrak{i}\mathfrak{l}\mathfrak{l}^{\rm loc}_{M,\omega_0}$, it is convenient to identify them with horizontal sections for a meromorphic connection on a trivial module over $E$. This is also a classical approach for general meromorphic parallelisms (see, for example, \cite[Lemma~3.2]{BlazquezCasale}). Indeed, let us denote by $T$ the torsion of the flat meromorphic connection $\nabla^0$ whose horizontal sections are the $\omega_0$-constant vector fields on $E$. 

\begin{Proposition}\label{connkillprop} Let $(E,\omega_0)$ be a meromorphic $(G,P)$-Cartan geometry on $(M,D)$.
\begin{itemize}\itemsep=0pt
\item[$(1)$] The sheaf $\mathfrak{k}\mathfrak{i}\mathfrak{l}\mathfrak{l}^{\rm loc}_{E,\omega_0}$ coincides with the sheaf $\ker(\nabla^{\rm rec}_{\omega_0})$ of horizontal sections for the reciprocal connection $\nabla^{\rm rec}$ defined by
\[ \nabla^{\rm rec}_X = \nabla^0_X + T(X,\cdot)\]
 for any local vector field $X$.
\item[$(2)$] The connection $\nabla^{\rm rec}_{\omega_0}$ is invariant by the $P$-linearization $({\rm d}r_b)_{b\in P}$ corresponding to the action of principal $P$-bundle.
\end{itemize}
\end{Proposition}
\begin{proof} (1) See \cite[Lemma 3.2]{BlazquezCasale}.  (2) This straightforwardly follows from the fact that the torsion of $\nabla_0$ is $P$-invariant by definition.
\end{proof}

\begin{Definition}\label{connkill} The \textit{Killing connection} of a meromorphic $(G,P)$-Cartan geometry $(E,\omega_0)$ on~$(M,D)$ is the meromorphic connection $(\mathcal{V},\nabla^{\omega_0})$ where $\mathcal{V}=\mathcal{O}_E\otimes \mathfrak{g}$ and
\[\nabla^{\omega_0}= \Phi_{\omega_0}^{-1} \nabla^{\rm rec},\]
 where $\Phi_{\omega_0}$ is the isomorphism of $\mathcal{O}_E\bigl(* \tilde{D}\bigr)$-modules between $TE\bigl(* \tilde{D}\bigr)$ and $\mathcal{V}\bigl(* \tilde{D}\bigr)$.
\end{Definition}

The sheaves $\mathfrak{k}\mathfrak{i}\mathfrak{l}\mathfrak{l}_{E,\omega_0}$ and $\mathfrak{k}\mathfrak{i}\mathfrak{l}\mathfrak{l}_{M,\omega_0}$ are respectively local systems on $E\setminus \tilde{D}$ and $M\setminus D$.

As explained in the introduction, our goal is to classify quasi-homogeneous meromorphic Cartan geometries $(E,\omega_0)$. This hypothesis is satisfied whenever the base manifold $M$ has only constant meromorphic functions (see \cite[Theorem 1.2]{Dumitrescu1}). In this case, there exist a point $x_0 \in M$ and $n$ independent germs of Killing vector fields for $(E,\omega_0)$ at $x_0$. We seek for a sufficient condition for these germs to come from global Killing vector fields, i.e., for the following property to be satisfied.

\begin{Definition}\label{extensionproperty} Let $(E,\omega_0)$ be a meromorphic $(G,P)$-Cartan geometry on a pair $(M,D)$. It satisfies the \textit{extension property of infinitesimal automorphisms} if and only the local system $\mathfrak{k}\mathfrak{i}\mathfrak{l}\mathfrak{l}_{M,\omega_0}$ on $M\setminus D$ extends as a local system $\mathfrak{k}\subset TM$ on $M$. \end{Definition}

\subsection[Distinguished foliations and (strongly) spiral meromorphic Cartan geometries]{Distinguished foliations and (strongly) spiral meromorphic Cartan \\ geometries}\label{sectionspiral}

We now isolate a subcategory of meromorphic Cartan geometries for which the extension property will be easier to obtain. For, we will use the straightforward generalization of \textit{distinguished curves} (see \cite[p.~112]{Cap} and \cite[Definition~4.16]{Sharpe} or definitions below).

Let $(E,\omega_0)$ be a meromorphic $(G,P)$-Cartan geometry on $(M,D)$, and $A\in \mathfrak{g}\setminus \{0\}$. Since~$\omega_0$ induces an isomorphism of meromorphic bundles between $TE\bigl(* \tilde{D}\bigr)$ and $\mathcal{O}_E\bigl(* \tilde{D}\bigr) \otimes \mathfrak{g}$, there exists a unique distribution of rank one (thus integrable) $\mathcal{T}_A \subset TE$ with the following property:
 \begin{equation}\label{definitionTA} \omega_0(\mathcal{T}_A) \subset \mathcal{O}_E\bigl(* \tilde{D}\bigr) A. \end{equation}
 We will call it the $A$-\textit{distinguished foliation} of $(E,\omega_0)$, and a leaf $\tilde{\Sigma}$ will be called a $A$-\textit{distinguished curve} for $(E,\omega_0)$.

Let $A\in \mathfrak{g}\setminus \{0\}$ and $\tilde{\Sigma}$ a $A$-distinguished curve for $(E,\omega_0)$. If $A\in \mathfrak{p}$, then $\tilde{\Sigma}$ is tangent to the kernel $\ker({\rm d}p)$ of the differential of the bundle map. If $A\not \in \mathfrak{p}$, $\tilde{\Sigma}\cap \big(E\setminus \tilde{D}\big)$ is transverse to this distribution. Hence, the restriction of $p$ to $\tilde{\Sigma}$ is a cover map from $\tilde{\Sigma}$ to its image $\Sigma\subset M$. In particular, if $\Sigma$ is simply connected, then it is a biholomorphism.

 \begin{Definition}\label{spiral} Let $A\in \mathfrak{g}$ and $(E,\omega_0)$ a meromorphic $(G,P)$-Cartan geometry.
 \begin{itemize}\itemsep=0pt
 \item[(1)] A $A$-\textit{spiral} for $(E,\omega_0)$ at $x_0 \in M$ is a complex \textit{smooth} curve $\Sigma$ embedded in $M$, containing~$x_0$ and such that $\Sigma \setminus D$ lifts to a $A$-distinguished curve in $E$.

 \item[(2)] A \textit{holomorphic $A$-spiral} is a $A$-spiral $\Sigma$ such that the lift $\tilde{\Sigma}$ as in (1) extends to a curve $\tilde{\Sigma}$ projecting onto $\Sigma$. \end{itemize}
 \end{Definition}

\begin{Definition}\label{taucondition} A meromorphic $(G,P)$-Cartan geometry $(E,\omega_0)$ on a pair $(M,D)$ is \textit{spiral} if the following is true. For any irreducible component $D_\alpha$ of $D$, there exist $x_0\in D_\alpha$ and a~spiral~$\Sigma$ for $(E,\omega_0)$ with $\Sigma \cap D = \{x_0\}$
\end{Definition}

Given any holomorphic foliation $\mathcal{T}_A$ of rank one on a complex manifold $E$, there exists an analytic subset ${\rm Sing}(\mathcal{T}_A)$ of codimension at least 2 with the following property (see \cite[p.~11]{Brunella}). For any $e_0 \in E\setminus {\rm Sing}(\mathcal{T}_A)$, there is a neighborhood $U$ of $e_0$ and a nonvanishing holomorphic vector field $Z \in TE(U)$ with
\begin{gather*}
 \mathcal{T}_A\vert_U = \mathcal{O}_U Z.
\end{gather*}
 We say, in this case, that $Z$ \textit{defines} $\mathcal{T}_A$ over $U$, and the leaves of $\mathcal{T}_A$ in $U$ are exactly the orbits of the local flow for $Z$. From this remark, we can easily deduce the first part of the following lemma.

 \begin{Lemma}\label{invariantfoliation}\samepage Let $E$ and $\mathcal{T}_A$ be as above. Let $\tilde{D}$ be any submanifold~$E$.
 \begin{itemize}\itemsep=0pt \item[$(1)$] The following assertions are equivalent:
 \begin{itemize}\itemsep=0pt
\item[$(i)$] $\tilde{D}$ is a union of leaves for $\mathcal{T}_A$.
\item[$(ii)$] Any local vector field $Z$ defining $\mathcal{T}_A$ is everywhere tangent to $\tilde{D}$.
\item[$(iii)$] For any local equation $z_1$ of $U\cap \tilde{D}$, and local vector field $Z$ defining $\mathcal{T}_A$ over $U$, the dimension
 \begin{gather}\label{tangency} \dim_{\mathbb{C}}\, \mathcal{O}_{E,e_0}/\langle \mathcal{L}_{Z}(z_1), z_1 \rangle_{e_0} \end{gather}
is never finite for $e_0 \in \tilde{D}\cap U$.
\end{itemize}
We say, in this case, that the submanifold $\tilde{D}\subset E$ is \textit{invariant} by $\mathcal{T}_A$.
\item[$(2)$] If $\tilde{D}$ is not invariant by $\mathcal{T}_A$, then there exists an Zariksi-dense subset $\tilde{W}\subset \tilde{D}\setminus {\rm Sing}(\mathcal{T}_A)$ with the following property: for any $e_0\in \tilde{W}$, there exists a leaf $\tilde{\Sigma}$ of $\mathcal{T}_A$ through $e_0$ satisfying~$\tilde{\Sigma}\cap \tilde{D} = \{e_0\}$.
\end{itemize}
 \end{Lemma}
\begin{proof}
(1) The equivalence between $(i)$ and $(ii)$ is clear from the above remark. The number~\eqref{tangency} is the order of tangency of $\mathcal{T}_A$ to $\tilde{D}$ at $e_0$ (see \cite[p.~13]{Brunella}).

(2) Let $z_1$ be a local equation for $\tilde{D}$, defined on an open subset $U\subset E$ where we can find a~holomorphic vector field $Z$ defining $\mathcal{T}_A$. Then the dimension \eqref{tangency} is zero except for a finite number of points in $\tilde{D}\cap U$ (see \cite[p.~13]{Brunella}). Complete $z_1$ into local coordinates $(z_1,\dots,z_n)$ on $U$, and decompose $Z=h \der{}{z_1} + Z'$ where $h$ is a holomorphic function on~$U$ and $Z'$ is a~holomorphic vector field on $U$ which belongs to the submodule spanned by $\der{}{z_2},\dots,\der{}{z_n}$. Since $\mathcal{L}_{Z'}(z_1)=0$, the previous fact implies that $h$ is not a multiple of~$z_1$. This means that $Z$ is generically transverse to $\tilde{D}\cap U$, and the leaves of $\mathcal{T}_A\vert_U$ are so, completing the proof of~(2).
 \end{proof}

 \begin{Definition}\label{strongtaucondition} Let $(E,\omega_0)$ be a meromorphic $(G,P)$-Cartan geometry on a pair $(M,D)$, and let $\big(\tilde{D}_\alpha\big)_{\alpha\in I}$ be the irreducible components of the divisor $\tilde{D}=p^{-1}(D)$. We say that $(E,\omega_0)$ is \textit{strongly spiral} if for any $\alpha \in I$, there exists $A\in \mathfrak{g}$ such that $\tilde{D}_\alpha$ is not invariant by $\mathcal{T}_A$.
 \end{Definition}

In what follows, we prove that in the case of a branched holomorphic Cartan geometry $(E,\omega_0)$ on a pair $(M,D)$ (see Definition~\ref{branchedCartan}), any spiral (see Definition~\ref{spiral}) crossing the polar locus~$D$ lifts to a distinguished curve (see paragraph before Definition~\ref{spiral}) crossing the lift $\tilde{D}$ of the polar locus. In this sense, the fact that $(E,\omega_0)$ is strongly spiral can be detected by considering the corresponding meromorphic geometric structure on $(M,D)$.

\begin{Lemma}\label{branchedfoliation} Let $(E,\omega_0)$ be a branched holomorphic $(G,P)$-Cartan geometry on $(M,D)$. Let $A\in \mathfrak{g}\setminus \mathfrak{p}$. Then the foliation $\mathcal{T}_A$ is transverse to $\ker({\rm d}p)$. In particular, for any complementary subspace $\mathfrak{g}_-$ of $\mathfrak{p}$, we obtain a holomorphic foliation of $E$ which is transverse to $\ker({\rm d}p)$. \end{Lemma}
\begin{proof}
Let $e_0\in E\setminus {\rm Sing}(\mathcal{T}_A)$, and $U$ be an open neighborhood of $e_0$ in $E\setminus {\rm Sing}(\mathcal{T}_A)$ equipped with coordinates $(z_1,\dots,z_N)$, with the property that $\der{}{z_{n+1}},\dots,\der{}{z_{N}}$ are sections of $\ker({\rm d}p)$ (i.e., vertical vector fields). Fix a basis $(\mathfrak{e}_1,\dots,\mathfrak{e}_N)$ of $\mathfrak{g}$ obtained by completing a basis $(\mathfrak{e}_{n+1},\dots,\mathfrak{e}_N)$ of $\mathfrak{p}$.

Since $(E,\omega_0)$ is a branched holomorphic Cartan geometry, and by the property $(i)$ of Definition~\ref{Cartan}, the matrix $Q=(q_{ij})_{i,j=1,\dots,N}$ of $\omega_0$ in the previous basis takes the form 
 \begin{align*} Q = \begin{pmatrix}
 A & 0 \\ B & C \end{pmatrix}, \end{align*}
 where
  \begin{itemize}\itemsep=0pt
\item $A$, $B$, $C$ are holomorphic matrices on $U$,
\item $C'=C^{-1}$ is a holomorphic matrix $U$.
\end{itemize}

The matrix of the $\omega_0$-constant vector fields associated with $(\mathfrak{e}_j)_{j=1,\dots,N}$ in $\big(\der{}{z_i}\big)_{i=1,\dots,N}$ is $Q^{-1}$ and therefore has the form
\begin{align*}
Q^{-1}= \begin{pmatrix} A_1 & 0 \\ B_1 & C_1\end{pmatrix} \end{align*}
 with
\[
CB_1+BA_1 =0.
\]
Consider any irreducible component $\tilde{D}_\alpha$ of $\tilde{D}=p^{-1}(D)$, and any $j\in \{n+1,\dots,N\}$. Since $C^{-1}$ is holomorphic on $U$, the above equation implies that for any vector in $\mathbb{C}^n$
\begin{align}\label{ordres}
 \operatorname{ord}_{\tilde{D}_\alpha \cap U}  \left(B_1 \begin{pmatrix} a_1\\ \vdots \\ a_n \end{pmatrix}\right)_j \geq \min_{i=1,\dots,n}
\operatorname{ord}_{\tilde{D}_\alpha \cap U}  \left( A_1\begin{pmatrix}a_1\\\vdots \\ a_n \end{pmatrix}\right)_i, \end{align}
 where the subscript stands for the $i$-th component.

We now interpret the inequality \eqref{ordres} geometrically. Let $Z$ be a holomorphic vector field defining $\mathcal{T}_A$ on $U$. Define $\left(\begin{smallmatrix}a_1\\ \vdots \\a_N\end{smallmatrix}\right)$ to be the coordinates of $A$ in the basis $(\mathfrak{e}_i)_{i=1,\dots,N}$. Thus
 \begin{gather*}
 Z= h \tilde{A} \end{gather*}
 with $h$ a meromorphic function on $U$ satisfying
 \begin{align*} \operatorname{ord}_{\tilde{D}_\alpha \cap U} (h)&{}= - \min_{j=1,\dots,N} \operatorname{ord}_{\tilde{D}_\alpha \cap U}\Bigg(\sum_{i=1}^{N} q^{-1}_{ji} a_i\Bigg)\\
 &= -\min\left(\min_{i'=1,\dots,n} \operatorname{ord}_{\tilde{D}_\alpha \cap U}\left(\!A_1 \begin{pmatrix} a_1 \\ \vdots \\ a_n \end{pmatrix}\right)_{i'}, \min_{j'=1,\dots,n} \operatorname{ord}_{\tilde{D}_\alpha \cap U} \left(\!B_1 \begin{pmatrix} a_1 \\ \vdots \\ a_n \end{pmatrix}\right)_{j'}\right).
 \end{align*}
Using \eqref{ordres}, we obtain that
\[\operatorname{ord}_{\tilde{D}_\alpha\cap U}(h) = - \min_{i'=1,\dots,n} \operatorname{ord}_{\tilde{D}_\alpha \cap U}\left(A_1 \left(\begin{matrix} a_1 \\ \vdots \\ a_n \end{matrix}\right)\right)_{i'}.\]

Decompose $Z=Z'+Z''$ with $Z'$ in the subsheaf of holomorphic vector fields spanned by $\big(\der{}{z_i}\big)_{i=1,\dots,n}$ and $Z''$ in the subsheaf $\ker({\rm d}p)$. Then the coordinates of $Z'$ (resp.\ $Z''$) in $\big(\der{}{z_i}\big)_{i=1,\dots,n}$ \big(resp.\ $\big(\der{}{z_i}\big)_{i=n+1,\dots,N}\big)$ are
\[
h A_1 \begin{pmatrix} a_1\\ \vdots \\ a_n \end{pmatrix}, \qquad \text{respectively},\qquad
h B_1 \begin{pmatrix} a_1 \\ \vdots \\ a_n \end{pmatrix} + h C_1 \begin{pmatrix} a_{n+1}\\ \vdots \\ a_N \end{pmatrix}.\]
 Using the above remark, and \eqref{ordres} again, we conclude that $Z'$ has order zero along $\tilde{D}_\alpha \cap U$ since $a_{i_0}\neq 0$ for some $i_0 \in \{1,\dots,n\}$, while $Z''$ is holomorphic. Since $Z'$ never vanishes on the regular part of $(E,\omega_0)$ (up to restriction of $U$), we conclude that $Z$ is nowhere tangent to $\ker({\rm d}p)$, concluding the proof.
\end{proof}

\begin{Proposition}\label{branchedtaucondition} Let $(E,\omega_0)$ be a branched holomorphic $(G,P)$-Cartan geometry on $(M,D)$. Suppose that no irreducible component $D_\alpha$ of $D$ is invariant by the spirals of $(E,\omega_0)$. Then $(E,\omega_0)$ is strongly spiral. \end{Proposition}

\begin{proof} Pick an irreducible component $\tilde{D}_\alpha$ of $\tilde{D}=p^{-1}(D)$. We have to prove that there exist $A\in \mathfrak{g}$ and a $A$-distinguished curve $\tilde{\Sigma}$ transverse to $\tilde{D}_\alpha$. But by the hypothesis, there exists a spiral $\Sigma$ of $(E,\omega_0)$ which is transverse to $D_\alpha = p\big(\tilde{D}_\alpha\big)$. Thus, there exist $A\in \mathfrak{g}\setminus \mathfrak{p}$ and a~$A$-distinguished curve $\tilde{\Sigma}$ with $p\big(\tilde{\Sigma}\big)= \Sigma \setminus D_\alpha$.

Remark that $\mathcal{T}_A$ is the kernel of an unique holomorphic $P$-principal connection on the restricted bundle $E\vert_{\Sigma\setminus D_\alpha}$, and $\tilde{\Sigma}$ is a horizontal section of this connection. By Lemma~\ref{branchedfoliation}, this holomorphic principal connection extends as a holomorphic principal connection on $E\vert_{\Sigma}$. Hence, $\tilde{\Sigma}$ extends as a horizontal section of $E\vert_{\Sigma}$, i.e., $p\big(\tilde{\Sigma}\big)=\Sigma$.
Up to restriction of $\Sigma$, we can assume that $\Sigma \cap D_\alpha =\{x_0\}$. Thus $\tilde{\Sigma}$ intersects $\tilde{D}_\alpha$ in some point in the fiber of $x_0$, concluding the proof.\looseness=1
\end{proof}

\section{Single-valued infinitesimal automorphisms of meromorphic \allowbreak parabolic geometries}\label{killparab}
A classical result in Riemannian geometry states that any Killing vector field $X$ for a Riemannian metric $g$ is a \textit{Jacobi field}: for any geodesic $\gamma$, its scalar product $g(X(\gamma(t)),\gamma'(t)))$ with the velocity of $\gamma$ is constant. There is a natural generalization of Riemannian metrics to the holomorphic category, and the corresponding objects are equivalent to torsionfree holomorphic affine connections preserving a holomorphic reduction to the orthogonal group. The holomorphic version of the previous result can be seen as a result on some torsionfree holomorphic affine Cartan geometries (see Corollary~\ref{equivalencetorsion}).

In this section, we will see a general result for meromorphic Cartan geometries (see Theorem~\ref{Lie-Cartan}). We then restrict to a subclass of geometries, namely the \textit{regular parabolic} ones (see Section~\ref{sectionparabolique}). For this category of meromorphic Cartan geometries, the previous result implies an extension theorem for the local system of infinitesimal automorphisms, which can be seen as a generalization of the above facts.

\subsection{Bott connections and infinitesimal automorphisms of Cartan geometries}\label{sectionBott}

Consider a complex Klein geometry $(G,P)$. Let $(M,D)$ be a complex pair of dimension $n\geq 1$, and $(E,\omega_0)$ be a meromorphic $(G,P)$-Cartan geometry on it. Fix $A \in \mathfrak{g} \setminus \{0\}$ and consider the holomorphic foliation $\mathcal{T}_A$ from \eqref{definitionTA}. To any such holomorphic foliation is associated a $\mathcal{T}_A$-partial holomorphic connection $\nabla^{\mathcal{T}_A}$ on $TE/\mathcal{T}_A$, the \textit{Bott-connection} of $\mathcal{T}_A$, defined as follow. Let $X$ be a holomorphic vector field on $U\subset E$, $[X]$ its class in $TE/\mathcal{T}_A(U)$, and $Z \in \mathcal{T}_A(U)$. Then
\begin{gather}\label{Bottconnection} Z\lrcorner \nabla^{\mathcal{T}_A}([X]) = [ [Z,X]_{TE} ]. \end{gather}
 Let $t\in \mathbb{C}$ and $V\subset U$ such that the flow $\phi=\phi^t_Z$ is well defined on $V$. Then clearly ${\rm d}\phi(\mathcal{T}_A) \subset \phi^* \mathcal{T}_A$, so $\phi$ induces a morphism $[{\rm d}\phi]$ of $\mathcal{O}_{V}$-modules defined by the commutative diagram
 \begin{gather}\label{dphiBott}
 \begin{split}
 & \xymatrix{ TV \ar[r]^-{{\rm d}\phi} \ar[d]_{[]} & \phi^* T\phi(V) \ar[d]^{\phi^* []} \\ TV/\mathcal{T}_A \ar[r]_-{[{\rm d}\phi]} & \phi^* T\phi(V)/\mathcal{T}_A.}
 \end{split} \end{gather}
 By the formula \eqref{Bottconnection}, the horizontal sections for $\nabla^{\mathcal{T}_A}$ are the $[X]$ which are invariant by the isomorphisms of holomorphic vector bundle $(\phi,[{\rm d}\phi])$ defined as before.

It will be more convenient to work with the images of meromorphic vector fields on $E$ under the isomorphism induced by the Cartan connection
 \begin{equation}\label{Phi}\Phi_{\omega_0}\colon\ TE\bigl(* \tilde{D}\bigr) \longrightarrow \mathcal{V}\bigl(* \tilde{D}\bigr),
 \end{equation}
 where $\mathcal{V}=\mathcal{O}_E\otimes \mathfrak{g}$. We will write
 \begin{equation}\label{K} \mathcal{K} = \Phi_{\omega_0}(\mathfrak{k}\mathfrak{i}\mathfrak{l}\mathfrak{l}_{E,\omega_0}) \end{equation}
 for the corresponding local system on $E\setminus \tilde{D}$. Clearly, the image of $\mathcal{T}_A\bigl(* \tilde{D}\bigr)$ is $\mathcal{V}_A=\mathcal{O}_E\bigl(* \tilde{D}\bigr) A$. The class of a section $s$ of $\mathcal{V}\bigl(* \tilde{D}\bigr)(U)$ (where $U\subset E$ is an open subset) in $\mathcal{V}/\mathcal{V}_A\bigl(* \tilde{D}\bigr)$ will be denoted by $[s]_{\mathcal{V}/\mathcal{V}_A}$. Since $\Phi_{\omega_0}$ induces an isomorphism of $\mathcal{O}_E$-modules between $TE/\mathcal{T}_A\bigl(* \tilde{D}\bigr)$ and $\mathcal{V}/\mathcal{V}_A\bigl(* \tilde{D}\bigr)$, for any $Z\in \mathcal{T}_A(U)$, the morphism $[{\rm d}\phi]$ defined by \eqref{dphiBott} corresponds to an isomorphism
 \begin{gather}\label{dphibarBott}\overline{{\rm d}\phi}\colon\ \mathcal{V}/\mathcal{V}_A\bigl(* \tilde{D}\bigr)\vert_V \longrightarrow \phi^* \mathcal{V}/\mathcal{V}_A\bigl(* \tilde{D}\bigr)\vert_{\phi(V)}
 \end{gather}
 and thus an isomorphism $\big(\phi,\overline{{\rm d}\phi}\big)$ of meromorphic bundles.

The isomorphism of meromorphic bundles $\Phi_{\omega_0}$ (see above) maps $\mathcal{T}_A\bigl(* \tilde{D}\bigr)$ to $\mathcal{V}_A\bigl(* \tilde{D}\bigr)$, and we denote by $\overline{\Phi}_{\omega_0}\colon TE/\mathcal{T}_A\bigl(* \tilde{D}\bigr) \longrightarrow \mathcal{V}/\mathcal{V}_A\bigl(* \tilde{D}\bigr)$ the isomorphism induced by $\Phi_{\omega_0}$. 

\begin{Lemma}\label{killBott} Let $s$ be a section of $\mathcal{K}$ on an open subset $U\subset E\setminus \tilde{D}$. Then its class $[s]_{\mathcal{V}/\mathcal{V}_A}$ is invariant by any isomorphism of meromorphic bundles $\big(\phi,\overline{{\rm d}\phi}\big)$ constructed as above. \end{Lemma}
\begin{proof} Let $X$ be any holomorphic vector field on $U\subset E$, and $[X]$ its class in $TE/\mathcal{T}_A$. By definition, for any $Z_A = h \tilde{A}$ (where $h$ is a meromorphic function on $U$ and $\tilde{A}=\omega_0^{-1}(A)$), we have
\begin{align*}
 0 = \big[\tilde{A},X\big]_{TE} = \frac{1}{h}[Z_A,X]_{TE} \mod \mathcal{T}_A\bigl(* \tilde{D}\bigr)(U).
 \end{align*}

In other words, the classes of ${\rm d}\phi(X)$ and $\phi^* X$ in $TE/\mathcal{T}_A\bigl(* \tilde{D}\bigr)$, well defined on $U\cap \phi(U)$, coincides, i.e., $s$ is invariant by $\big(\phi,\overline{{\rm d}\phi}\big)$. \end{proof}

Now, we suppose $M$ to be simply connected. We wish to prove the extension property for~$(E,\omega_0)$ (see Definition~\ref{extensionproperty}). We will use the following general fact on meromorphic Cartan geometries.

\begin{Theorem}\label{Lie-Cartan} Let $(G,P)$ be a complex Klein geometry, and $(M,D)$ be a pair with $\dim(M)=\dim(G/P)$. Let $(E,\omega_0)$ be a meromorphic $(G,P)$-Cartan geometry on $(M,D)$, $\mathcal{V}=\mathcal{O}_E\otimes \mathfrak{g}$ and~$\mathcal{K}$ the sheaf defined as in \eqref{K}. Let $x_0 \in D$ belonging to the smooth part of an unique irreducible component $D_\alpha$, and $\tilde{D}_\alpha = p^{-1}(D_\alpha)$.

Suppose that there exist $A\in \mathfrak{g}\setminus \mathfrak{p}$ and a $A$-distinguished curve $\tilde{\Sigma}$ for $(E,\omega_0)$ such that $\tilde{\Sigma}\cap \tilde{D}_\alpha= \{e_0\}$ for some point $e_0$ in the fiber of $x_0$.

Denote by $\mathcal{T}_A$ the distinguished foliation of $E$ \eqref{definitionTA}, as well as $\mathcal{V}_A=\Phi_{\omega_0}(\mathcal{T}_A)$ \textup{(}see \eqref{Phi}{\rm )}. Then there exists a neighborhood $U$ of $x_0$ with the following properties:
 \begin{itemize}\itemsep=0pt
 \item[$(1)$] Let $s$ be a section of $\mathcal{K}$ on an open subset $V\subset p^{-1}(U\setminus D)$. Then the class $[s]_{\mathcal{V}/\mathcal{V}_A}$ of $s$ in~$TE/\mathcal{T}_A$ extends as a (single-valued) section of $TE/\mathcal{T}_A$ on $p^{-1}(U\setminus D)$.
 \item[$(2)$] The image under $\Phi_{\omega_0}$ $($see \eqref{Phi}$)$ of the section defined in~$(1)$ is the restriction of a section of $\mathcal{V}/\mathcal{V}_A(* D)$ over $p^{-1}(U)$.
\item[$(3)$] Suppose moreover that $(E,\omega_0)$ is holomorphic branched on $(M,D)$. Then the above section lies in $\mathcal{V}/\mathcal{V}_A(p^{-1}(U))$. \end{itemize}
\end{Theorem}
\begin{proof}
(1) It is clear from the definition that the image $\mathcal{K}_A=\pi_{\mathcal{V}/\mathcal{V}_A}(\mathcal{K})$ of the local system~$\mathcal{K}$ on $E\setminus \tilde{D}$ is a local system on $E\setminus \tilde{D}$. Moreover, $\mathcal{K}$ is $P$-equivariant, and the same remains valid for $\mathcal{K}_A$. We thus have to prove that $\mathcal{K}_A$ is a constant sheaf on an open subset $\tilde{U}\setminus \tilde{D}$ where $\tilde{U}\subset E$ containing some $e_0 \in p^{-1}(x_0)$. It suffices to prove that there exists a simply connected $\tilde{U}$, with $\tilde{U}\cap \tilde{D}$ simply connected, and $e'_0 \in \tilde{U}\cap \tilde{D}$ such that $\mathcal{K}_A$ is a constant sheaf in a neighborhood of~$e'_0$.

The hypothesis says that $\tilde{D}_\alpha$ is not invariant by $\mathcal{T}_A$. By Lemma~\ref{invariantfoliation}, there exists an open neighborhood $\tilde{U}_0$ of some point $e_0 \in p^{-1}(x_0)$ such that a generic leaf of $\mathcal{T}_A\vert_{U_0}$ intersects $\tilde{D}\cap U_0$ in exactly one point. Pick $e'_0$ such that there exists a leaf $\tilde{\Sigma}$ of $\mathcal{T}_A$ with $\tilde{\Sigma}\cap \tilde{D} = \{e'_0\}$. Pick $e\in \tilde{\Sigma}\setminus \tilde{D}$, and a simply connected open neighborhood $V$ of $e$ equipped with a basis $(s_1,\dots,s_r)$ of $\mathcal{K}$.

By Lemma~\ref{killBott}, the family $\overline{{\rm d}\phi}([s_1]_{\mathcal{V}/\mathcal{V}_A}),\dots,\overline{{\rm d}\phi}([s_r]_{\mathcal{V}/\mathcal{V}_A})$, where $\overline{{\rm d}\phi}$ is the morphism \eqref{dphibarBott}, is a~basis of $\phi^*\mathcal{K}_A\big(\phi(V)\setminus \tilde{D}\big)$. Recall that $\phi$ were defined as the flow of a non-vanishing (hence complete) holomorphic vector field $Z$ of $ \mathcal{T}_A$ on $\tilde{U}$, for some time $t\in \mathbb{C}$. In particular, $\tilde{\Sigma}$ is the union of images $\phi(e)$ for various $t$. By the previous remark, there exists such a $t\in \mathbb{C}$ with $\phi(e)=e'_0$. Then $\mathcal{K}_A$ is a constant sheaf when restricted to $\phi(V)\setminus \tilde{D}$, and $\phi(V)$ is a~neighborhood of $e'_0$. By the above remarks, this concludes the proof. 

(2) Since \eqref{dphibarBott} is an isomorphism of meromorphic bundles, we have proved in~(1) that the local system $\mathcal{K}_A=\pi_{\mathcal{V}/\mathcal{V}_A}(\mathcal{K})$ extends as a constant sheaf, included in $\mathcal{V}/\mathcal{V}_A\bigl(* \tilde{D}\bigr)\vert_{p^{-1}(U)}$ since $\big(\phi,\overline{{\rm d}\phi}\big)$ is an automorphism of meromorphic bundles for $\mathcal{V}/\mathcal{V}_A$.

(3) The meromorphic Cartan geometry $(E,\omega_0)$ is holomorphic branched on $(M,D)$ if and only if $\Phi_{\omega_0}(TE)\subset \mathcal{V}$. Suppose this is the case. Since the automorphism of meromorphic bundles $(\phi,[{\rm d}\phi])$ of $TE/\mathcal{T}_A$ defined before \eqref{dphiBott} is an automorphism of holomorphic vector bundles. Since $\Phi_{\omega_0}(TE)$ and $\mathcal{V}$ coincides when restricted to $p^{-1}(U\setminus D)$, we obtain that the image of $\mathcal{V}/\mathcal{V}_A\vert_{p^{-1}(U\setminus D)}$ under the $\overline{{\rm d}\phi}$ lies in $\mathcal{V}/\mathcal{V}_A\vert_{\phi(p^{-1}(U\setminus D)})$, where $\phi(p^{-1}(U\setminus D))$ is a neighborhood of~$e'_0$ by construction. This proves the assertion.
\end{proof}

\subsection{ Regular meromorphic parabolic geometries}\label{sectionparabolique}
In this subsection we recall elementary facts about an important class of Cartan geometries, the \textit{parabolic} ones. The reader can consult \cite{Cap} for further details.

 A \textit{complex parabolic Klein geometry} is a complex Klein geometry $(G,P)$, where $G$ is a complex semi-simple Lie group and $P$ is a parabolic subgroup. A \textit{meromorphic parabolic geometry} is a~meromorphic $(G,P)$-Cartan geometry for some complex parabolic Klein geometry. We refer the reader to \cite{Cap} for a detailed introduction.

 With the subgroup $P$ is associated a grading $(\mathfrak{g}_i)_{i\in \mathbb{Z}}$ of the Lie algebra $\mathfrak{g}={\rm Lie}(G)$, meaning%
  \begin{gather}\label{conditiongraded}[\mathfrak{g}_{i_1},\mathfrak{g}_{i_2}]_\mathfrak{g} \subset \mathfrak{g}_{i_1+i_2}\end{gather}
 for any indices $i_1,i_2\in \mathbb{Z}$. We call it the \textit{parabolic filtration} associated with $P$. It induces a~grading of any representation of $G$, in particular \smash{$\mathbb{W}=\bigl(\bigwedge^2 \mathfrak{g}_-^*\bigr)\otimes \mathfrak{g} $} is graded by \textit{homogeneous degrees} $\mathbb{W}_l$, and we denote by $\pi_l$ the corresponding projections.

 The \textit{parabolic degree} of $(G,P)$ is the smallest positive integer $k\geq 1$ such that $\mathfrak{g}_i =\{0\}$ for any $\vert i\vert>k$. By \eqref{conditiongraded}, the subspace $\mathfrak{p}={\rm Lie}(P)$ and the subspace
 \begin{gather*}
 \mathfrak{g}_- = \bigoplus_{i=1}^{k}  \mathfrak{g}_{-i} \end{gather*}
 are clearly two subalgebras of $\mathfrak{g}$. For any $i\in \{-k,\dots,k\}$, we will denote $\mathfrak{g}^i = \bigoplus_{i'\geq i} \mathfrak{g}_{i'}$, inducing a filtration $(\mathfrak{g}^i)_{i=-k,\dots,k}$ of $\mathfrak{g}$.

 By \cite{Chevalley}, we can and do always pick a basis $\big(\mathfrak{e}^i_j\big)_{\substack{i=-k,\dots,k \\ j=1,\dots,n_i}}$ of $\mathfrak{g}$, such that $\big(\mathfrak{e}_j^i\big)_{j=1,\dots,n_i}$ is a basis for $\mathfrak{g}_i$ for any $i\in \{-k,\dots,k\}$, and $\big[\mathfrak{e}^{i_1}_{j_1},\mathfrak{e}^{i_2}_{j_2}\big]_\mathfrak{g}$ is a vector in $\mathbb{Z}\mathfrak{e}^{i_1+i_2}_{j}$ for some $j\in \{1,\dots,n_{i_1+i_2}\}$. We will refer to it as a \textit{graded basis} of $\mathfrak{g}$ for $(G,P)$.

 The homogeneous space $G/P$ associated with a complex parabolic Klein geometry $(G,P)$ bears the following holomorphic geometric structure. Its tangent bundle is filtered by subbundles $\big(T^{-i}G/P\big)_{i=1,\dots,k}$ where $T^{-i}G/P$ is the projection of $\omega_G^{-1}\big(\mathfrak{g}^i\big)$ through the tangent map $Tp_{G/P}$ of the projection $p_{G/P}\colon G \longrightarrow G/P$. The Lie bracket of holomorphic vector fields on $G/P$ induces a Lie bracket on the vector space of local sections of the corresponding graded bundle $\operatorname{gr}(TG/P)$. The Lie algebra bundle thus obtained is locally isomorphic to $(U\times \mathfrak{g}_-, [\,,\,]_{\mathfrak{g}_-})$.

 The \textit{regular meromorphic parabolic geometries} are the infinitesimal versions of this model. More precisely, these are meromorphic $(G,P)$-Cartan geometries $(E,\omega_0)$ on $(M,D)$ for which the homogeneous component $\pi_l(k_{\omega_0})$ of degree $l$ of the Cartan curvature vanishes identically whenever $l\leq 0$ (see above). This amounts to the following property. Let $T^{-i}M(* D)$ be the image of $\omega_0^{-1}\big(\mathfrak{g}^{-i}\big)\bigl(* \tilde{D}\bigr)$ through $Tp$. This gives a filtration of $TM\bigl(* \tilde{D}\bigr)$, and $(E,\omega_0)$ is regular if and only if the Lie bracket of vector fields on $M$ induces a structure of Lie algebras bundle on the graded $\operatorname{gr}(TM\setminus D)$, locally isomorphic to $(U\times \mathfrak{g}_-,[\,,\,]_{\mathfrak{g}_-})$.

\subsection{Affine and degree one parabolic models}
We now apply Theorem~\ref{Lie-Cartan} to prove the extension property for infinitesimal automorphisms of some holomorphic branched $(G,P)$-Cartan geometries which are spiral (see Definition~\ref{taucondition}). More precisely, we let~$(G,P)$ be either
\begin{itemize}\itemsep=0pt
 \item[$\bullet$] A complex parabolic Klein geometry of dimension $n\geq 2$ and degree $k=1$ (see Section~\ref{sectionparabolique}), with $G$ a complex simply connected simple Lie group. In this case we let $\mathfrak{g}=\mathfrak{g}_{-1}\oplus \mathfrak{g}_0\oplus \mathfrak{g}_1$ be the graded Lie algebra defined in Section~\ref{sectionparabolique}.
\item[$\bullet$] Or the complex affine Klein geometry $(G,P)$ of dimension $n\geq 2$. In this case we let $\mathfrak{g}=\mathfrak{g}_{-1}\oplus \mathfrak{g}_0$ be the decomposition of the Lie algebra of $G$ where $\mathfrak{g}_0={\rm Lie}(P)$ and $\mathfrak{g}_{-1}$ the (abelian) Lie algebra of the infinitesimal generators for the translations in $G$. By convention we let $\mathfrak{g}_{1}=\{0\}$. \end{itemize}

These two kind of models are of special interest because in both cases the Levi subgroup $G_0=\exp _{G}(\mathfrak{g}_0)\subset P$ acts transitively on the lines in $\mathfrak{g}_{-1}$. This is clear for the second one, since the corresponding action is the standard representation of ${\rm GL}_n(\mathbb{C})$. For the first one, remark that $P=G_0 \ltimes \exp _G(\mathfrak{g}_1)$, and $[\mathfrak{g}_1,\mathfrak{g}_{-1}]_\mathfrak{g} \subset \mathfrak{g}_0$ by \eqref{conditiongraded}. Hence, since $\mathfrak{g}$ is simple, $[\mathfrak{g}_0,\mathfrak{g}_{-1}]_\mathfrak{g}$ must contain $\mathfrak{g}_{-1}$ (otherwise $[\mathfrak{g},\mathfrak{g}_{-1}]_{\mathfrak{g}}$ would be a non-trivial ideal in $\mathfrak{g}$). Again by \eqref{conditiongraded} is contained in $\mathfrak{g}_{-1}$. So it is in fact equal to $\mathfrak{g}_{-1}$. This implies that the action of $G_0$ on $\mathfrak{g}_{-1}$ is open.

By the above remarks, for any $A\in \mathfrak{g}_{-1}\setminus \{0\}$, there exist a basis $(\mathfrak{e}_i)_{i=1,\dots,n}$ of the abelian subalgebra $\mathfrak{g}_{-1}\subset \mathfrak{g}={\rm Lie}(G)$ and $n$ elements $b_1,\dots,b_n \in P$ such that{\samepage \begin{equation}\label{pi} \operatorname{Ad}\big(b_i^{-1}\big)[A]\in \mathbb{C} \mathfrak{e}_i \end{equation} for $1\leq i\leq n$.}

\begin{Lemma}\label{lemmageodesicaffine} Let $(E,\omega_0)$ be a holomorphic branched $(G,P)$-Cartan geometry on $(M,D)$, and let $x_0 \in D$ belonging to the smooth part of an unique irreducible component $D_\alpha$. Suppose there exists a $A$-spiral $\overline{\Sigma}$ for $(E,\omega_0)$ with $\overline{\Sigma}\cap D=\{x_0\}$ and $A\neq 0$.

Then there exists an analytic subset $S_\alpha$ of codimension at least one in $D_\alpha$ with the following property. For any $x'_0 \in D_\alpha \setminus S_\alpha$, and for any $1\leq i\leq n$, there exists a holomorphic $\mathfrak{e}_i$-spiral $\overline{\Sigma}_i$ for $(E,\omega_0)$ with $\overline{\Sigma}_i \cap D =\{x'_0\}$.
\end{Lemma}

\begin{proof} By Proposition~\ref{branchedtaucondition}, $\overline{\Sigma}$ is a $A$-holomorphic geodesic for $A\in \mathfrak{g}\setminus \mathfrak{p}$. Thus, there exist $e_0 \in p^{-1}(x_0)$ and a $A$-distinguished smooth complex curve $\tilde{\Sigma}$ with $p\big(\tilde{\Sigma}\big)=\overline{\Sigma}$ and $\tilde{\Sigma} \cap \tilde{D}_\alpha =\{e_0\}$. Let $b_1,\dots,b_n \in P$ as in \eqref{pi}. By equivariance of $\omega_0$, for any $1\leq i \leq n$, the $b_i$-translated $\tilde{\Sigma}_i$ of~$\tilde{\Sigma}$ is a $\mathfrak{e}_i$-distinguished smooth complex curve, with $\tilde{\Sigma}_i \cap \tilde{D}_\alpha = \{e_0 \cdot b_i\}$. This means that for any $i\in \{1,\dots,n\}$, $\tilde{D}_\alpha$ is not invariant by $\mathcal{T}_{\mathfrak{e}_i}$. By Lemma~\ref{invariantfoliation}, this implies that for any such~$i$, there exists a Zariski open-dense subset $W_i$ of $\tilde{D}_\alpha$ such that, for any $e'_0 \in W_i$, there exists a~leaf~$\tilde{\Sigma}_i$ of~$\mathcal{T}_{\mathfrak{e}_i}$ transverse to $\tilde{D}_\alpha$ at $e'_0$.

Consider the Zariksi-dense open subset in $D_\alpha$ defined by
 \begin{gather*} W_\alpha= p\bigg(\bigcap_{i=1}^{n} W_i\bigg). \end{gather*}
 By definition, for any $e'_0 \in p^{-1}(W_\alpha)$, and any $1\leq i \leq n$ there is a $\mathfrak{e}_i$-distinguished smooth curve~$\tilde{\Sigma}_i$ with $\tilde{\Sigma}_i \cap \tilde{D}=\{e'_0\}$. The corresponding projections $\Sigma_i$ through $p$ are holomorphic $\mathfrak{e}_i$-spirals for $(E,\omega_0)$, and the proof is thus achieved by considering $S_\alpha=D_\alpha\setminus W_\alpha$.
\end{proof}

We use the proof of the Deligne's extension theorem (see \cite[Chapter~II, Theorems~5.2 and~5.4]{Deligne}), to prove a lemma which enables us to manage the codimension one analytic subset $S=\bigcup_{\alpha\in I} S_\alpha$ in~$D$ from Lemma~\ref{lemmageodesicaffine}.

\begin{Lemma}\label{lemmalocalmonodromy} Let $M$ be a complex manifold and $D$ a divisor of $M$. Let $(\mathcal{E},\nabla)$ be a meromorphic connection on $M$ with pole at $D$.
Suppose the existence of an analytic subset $S$ of codimension at least one in $D$ with the following property. For any $x_0\in D\setminus S$, there exists an open neighborhood~$U_{x_0}$ of~$x_0$ in~$M$ such that the restriction of $\ker  \nabla$ to $U_{x_0}\setminus D$ is a constant sheaf.
Then $\ker  \nabla$ is the restriction of a local system on~$M$.
\end{Lemma}

\begin{proof} First remark that, by definition of a local system, the assertion to be proved is equivalent to the fact that for any $x_0\in D$, there exists an open neighborhood $U$ in $M$ such that $\ker  \nabla$ is a~constant sheaf on $U\setminus D$. Since this is satisfied for $x_0 \in D\setminus S$ by the hypothesis, we will prove it for $x_0\in S$. The proof then follows from the arguments used in \cite[Chapter~II, Theorems~5.2 and~5.4]{Deligne}. We write the details to avoid additional definitions.

By the Hironaka's desingularization theorem (see, for example, \cite[Theorem 7.13]{GrauertPeternellRemmert}), there exists a holomorphic map $\pi\colon M'\longrightarrow M$ from a complex manifold $M'$ of the same dimension as~$M$ and a divisor $D'$ of $M'$ with
\begin{itemize}\itemsep=0pt
 \item[(i)] The restriction $\pi|_{M'\setminus D'}\colon M'\setminus D' \longrightarrow M\setminus D$ is a biholomorphism.
 \item[(ii)] $D'$ is a \textit{normal crossing} divisor, that is for any $x_0\in D'$ there exits an open neighborhood~$U$ of $x_0$ in $M'$, coordinates $(z_1,\dots,z_n)$ on $U$ and indices $i_1,\dots,i_r\in \{1,\dots,n\}$ such that
 \begin{gather*}
 D'\cap U= \bigcup_{j=1}^{r} \{z_{i_j}=0\}.\end{gather*}
\end{itemize}

By the equivalence of categories between local systems and flat connections, the pullback $\pi^* \ker  \nabla$ is the local system of horizontal sections for the pullback of meromorphic connection $(\mathcal{E}',\nabla'):=(\pi^* \mathcal{E},\pi^\star \nabla)$ on $M'$, with poles at $D'$. By the property $(i)$ above, $\ker  \nabla$ is the restriction of a local system on $M$ if and only if $\ker  \nabla'$ is the restriction of a local system on $M'$. Hence, we assume from now that $D$ satisfies $(ii)$, i.e., is a normal crossing divisor.

So let $x_0\in S$, and $U$, $(z_1,\dots,z_n)$ as in $(ii)$, and pick $x\in U\setminus D$. Hence, up to restricting $U$, the fundamental group $\pi_1(U\setminus D,x)$ is spanned by generators $\gamma_1,\dots,\gamma_r$, such that for $j=1,\dots,r$, $\gamma_j$ is contained in an open subset which does not intersect $\{z_{i_{j'}}=0\}$ for $j'\neq j$. Moreover, these generators commute. Denote by $B_1,\dots,B_r$ the image of these generators under the monodromy of $\ker \nabla$ at $x$, and let $A_1,\dots,A_r\in \mathfrak{g}\mathfrak{l}_r(\mathbb{C})$ be matrices such that $B_j=\exp(A_j)$ for $j=1,\dots,r$, where we have fixed an isomorphism $\varphi_x$ between $\ker \nabla(x)$ and $\mathbb{C}^r$. Consider the holomorphic connection $\big(\mathcal{O}_{U\setminus D}^{\oplus r},\nabla_0\big)$ on $U\setminus D$, where
 \begin{gather*} \nabla_0 = {\rm d}+ \sum_{j=1}^{r} \frac{{\rm d}z_{i_j}}{z_{i_j}} \otimes A_j,
 \end{gather*}
 where ${\rm d}$ stands for the tensor product of the de Rham derivative and the identity on $\mathcal{O}_{U\setminus D}^{\oplus r}$. Then we can compute explicitly the horizontal sections of $\nabla_0$ (see \cite[p.~93]{Deligne}) to see that the image of each $\gamma_j$ under the monodromy of $\ker  \nabla_0$ and $\ker  \nabla$ at $x$ are conjuguated through the fixed isomorphism $\varphi_x$ as before. By the equivalence of categories between local systems and monodromy \cite[Chapter~I, Theorem~2.17]{Deligne}, we thus have an isomorphism of $\underline{\mathbb{C}}_{U\setminus D}$-sheaves
 \begin{gather}\label{isokernabla0} \varphi\colon\ \ker  \nabla_0 \overset{\sim}{\longrightarrow} \ker  \nabla. \end{gather}
 Moreover, since the $B_j$ and thus the $A_j$ commute, up to a linear change in $\mathbb{C}^r$, the $A_j$ have Jordan form.

Suppose there exists $j\in \{1,\dots,r\}$ such that $A_j$ have a non-integer eigenvalue $\alpha$. Then by the explicit form of the horizontal sections of $\nabla_0$ \cite[p.~93]{Deligne}, there is a section $s_0$ of $\ker  \nabla_0$ on an open dense subset of $U\setminus D$ of the form
\[s_0(z_1,\dots,z_n)=z_{i_j}^\alpha s'_0(z_1,\dots,z_n)\]
with $s'_0$ a $\mathbb{C}^r$-valued holomorphic function defined on $U$ and $\log$ the principal determination of the logarithm on $z_{i_j}(U)$. In particular, for any $x_0\in D\cap U$ with $z_{i_j}(x_0)=0$, and any open neighborhood $V\subset U$ of $x_0$, $\ker  \nabla_0|_{V\setminus D}$ is not a constant sheaf. In view of \eqref{isokernabla0}, this contradicts the hypothesis on $\nabla$. Hence all the eigenvalues of the $A_j$ are integers.

Suppose that there exists $j\in \{1,\dots,r\}$ with $A_j$ having a Jordan bloc of size at least two. Then, again by the explicit form for the horizontal sections of $\nabla_0$, there is a section $s_0$ of $\ker  \nabla_0$ of the form
\[s_0(z_1,\dots,z_n)=z_{i_j}^k (s'_1(z_1,\dots,z_n)+\log(z_{i_j})s'_2(z_1,\dots,z_n)),\]
 where $k$ is an integer, $s'_1$, $s'_2$ are two $\mathbb{C}^r$-valued holomorphic functions on $U$ and $\log$ is as before. Again, $\ker  \nabla_0$ is not a constant sheaf in a neighborhood of any point $x_0\in \{z_{i_j}=0\}$, contradicting the hypothesis on $\nabla$ in view of \eqref{isokernabla0}.

As a consequence, all the $A_j$ are diagonalizable endomorphisms with integer eigenvalues, so that $B_j=\exp(A_j)$ is the identity for any $j\in \{1,\dots,r\}$. Hence $\ker  \nabla_0$ is a constant sheaf on~$U\setminus D$, and in view of \eqref{isokernabla0} the same holds for $\ker  \nabla$, concluding the proof.
\end{proof}

\begin{Corollary}\label{proprextaffine} Let $(G,P)$ be as above, and $(E,\omega_0)$ be a holomorphic branched $(G,P)$-Cartan geometry on $(M,D)$, with $M$ simply connected. Suppose it is spiral $($Definition~{\rm \ref{taucondition})}. Then
\begin{itemize}\itemsep=0pt
\item[$(1)$] $(E,\omega_0)$ satisfies the extension property for the infinitesimal automorphisms.
\item[$(2)$] Any section $s$ of $\ker\big(\nabla^\kappa_{\omega_0}\big)(U)$ $($where $\nabla^\kappa_{\omega_0}$ is the Killing connection, see Definition~{\rm \ref{connkill}}, and $U\subset E$ is an open subset of $M)$ is a section of $\mathcal{V}(U)$.
\end{itemize}\end{Corollary}
\begin{proof} (1) Let $(D_\alpha)_{\alpha\in I}$ stand for the irreducible components of the divisor $D$. By definition of a spiral meromorphic Cartan geometry and Lemma~\ref{lemmageodesicaffine}, for any $\alpha\in I$, there exists an analytic subset $S_\alpha$ of codimension at least one in $D_\alpha$ such that for any $x_0\in D_\alpha\setminus S_\alpha$, and any $i=1,\dots,n$, there exists a $\mathfrak{e}_i$-spiral $\Sigma_i$ with $\Sigma_i\cap D_\alpha=\{x_0\}$ where $(\mathfrak{e}_i)_{i=1,\dots,n}$ is the basis of $\mathfrak{g}_{-1}$ defined before. Moreover, by Proposition~\ref{branchedtaucondition}, $\Sigma_i$ is a holomorphic spiral in the sense that it lifts to $E$ as a~distinguished curve (intersecting $\tilde{D}_\alpha$ at some $e_{0,i}\in p^{-1}(x_0)$).

The Lemma~\ref{lemmalocalmonodromy} implies that it is sufficient to prove that for any $x_0$ as above, there exists an open neighborhood $U$ of $x_0$ in $M$ such that the restriction $\ker\big(\nabla^\kappa_{\omega_0}\big)|_{U'\setminus D_\alpha}$ is a constant sheaf.

Thus, pick $\alpha \in I$, $x_0 \in D_\alpha\setminus S_\alpha$ and let us prove the existence of an open neighborhood~$U$ of~$x_0$ in $M$ such that $\ker\big(\nabla^\kappa_{\omega_0}\big)|_{U\setminus D}$ is a constant sheaf. By the first remark, for any $1\leq i\leq n$, there is a holomorphic $\mathfrak{e}_i$-spiral $\Sigma_i$ with $\Sigma_i \cap D=\{x_0\}$. More precisely, the proof of the lemma implies the existence, for $i\in \{1,\dots,n\}$ fixed, of $e_0\in p^{-1}(x_0)$ such that the $\mathfrak{e}_i$-distinguished curve $\tilde{\Sigma}_i$ projecting onto $\Sigma_i$ satisfies $\tilde{\Sigma}_i \cap \tilde{D}=\{e_0\}$. Using the~Theorem~\ref{Lie-Cartan} for each geodesic, we obtain neighborhoods $U_i$ of $x_0$ such that the restriction of the local system $\pi_{\mathcal{V}/\mathcal{V}_{\mathfrak{e}_i}}\bigl(\ker\bigl(\nabla^\kappa_{\omega_0}\bigr)\bigr)$ to $p^{-1}(U_i)$ is a constant sheaf.

Let $U=\bigcap_{i=1}^{n} U_i$. Since $\mathfrak{e}_1$, $\mathfrak{e}_2$ are independent vectors of $\mathfrak{g}$, the morphism of $\mathcal{O}_E$-modules
\begin{align}\label{piV1V2}
 \pi_{\mathcal{V}/\mathcal{V}_{\mathfrak{e}_1}} \oplus \pi_{\mathcal{V}/\mathcal{V}_{\mathfrak{e}_2}}\colon \ \mathcal{V}\bigl(* \tilde{D}\bigr) \longrightarrow \mathcal{V}/\mathcal{V}_{\mathfrak{e}_1}\bigl(* \tilde{D}\bigr) \oplus \mathcal{V}/\mathcal{V}_{\mathfrak{e}_2}\bigl(* \tilde{D}\bigr) \end{align}
 is an isomorphism onto its image. Thus, it restricts to $\ker\big(\nabla^\kappa_{\omega_0}\big)$ as an isomorphism of $\mathbb{C}$-sheaves onto its image, a subsheaf of the local system $\pi_{\mathcal{V}/\mathcal{V}_{\mathfrak{e}_1}}\!\big( \ker\big(\nabla^\kappa_{\omega_0}\big)\big) \oplus \pi_{\mathcal{V}/\mathcal{V}_{\mathfrak{e}_2}}\!\big(\ker\big(\nabla^\kappa_{\omega_0}\big)\big)$. By the above remark, this local system is a constant sheaf when restricted to $p^{-1}(U)$. Thus, the same is true for $\ker\big(\nabla^\kappa_{\omega_0}\big)$, i.e., $(E,\omega_0)$ satisfies the extension property for the infinitesimal automorphisms.

(2) Since $(E,\omega_0)$ is a branched holomorphic Cartan geometry, we can apply the point~(3) of~Theorem~\ref{Lie-Cartan} to $A=\mathfrak{e}_1$ and $A=\mathfrak{e}_2$. We obtain that the image of $\ker\big(\nabla^{\kappa}_{\omega_0}\big)$ through $\pi_{\mathcal{V}/\mathcal{V}_{\mathfrak{e}_1}}$ and $\pi_{\mathcal{V}/\mathcal{V}_{\mathfrak{e}_2}}$ respectively extends as subsheaves of $\mathcal{V}/\mathcal{V}_{\mathfrak{e}_1}$ and $\mathcal{V}/\mathcal{V}_{\mathfrak{e}_2}$ on $E$. Since the morphism~\eqref{piV1V2} clearly restricts to a morphism between $\mathcal{V}$ and $\mathcal{V}/\mathcal{V}_{\mathfrak{e}_1} \oplus \mathcal{V}/\mathcal{V}_{\mathfrak{e}_2}$, this proves the assertion.\looseness=-1
\end{proof}

\begin{Remark}\label{remarkmeromorphicproprext} The conclusion of point (1) in Corollary~\ref{proprextaffine} remains valid if we consider a meromorphic $(G,P)$-Cartan geometry $(E,\omega_0)$ on $(M,D)$, but replacing the \textit{spiral} by \textit{strongly spiral} (see Definition~\ref{strongtaucondition}). The conclusion of point (2) remains true if $\mathcal{V}$ is replaced by $\mathcal{V}(* D)$.
\end{Remark}

\subsection{Parabolic geometries of higher degree}
Now, we let $(G,P)$ be a complex parabolic Klein geometry of degree $k>1$, and denote by $\mathfrak{g}_{-k}\oplus \dots \oplus \mathfrak{g}_0 \oplus \dots\oplus \mathfrak{g}_k$ the parabolic filtration. We refer the reader to \cite{Cap} for the definitions and a~complete introduction on this subject.

 The group $P$ no longer acts transitively on $\mathbb{P}(\mathfrak{g}/\mathfrak{p})$. Instead, we use a result of \cite[Theorem, p.~13]{Cap4} which implies the following lemma.

\begin{Lemma}\label{caractersiticparabolic} Let $(E,\omega_0)$ be a regular meromorphic $(G,P)$-Cartan geometry on a pair $(M,D)$. Then there exists a morphism of $\mathbb{C}$-sheaves
\begin{gather*}
 \mathcal{L} \colon\ \mathcal{V}_{\mathfrak{g}_{-k}}\bigl(* \tilde{D}\bigr) \longrightarrow \mathcal{V}\bigl(* \tilde{D}\bigr) \end{gather*}
 with the following properties:
\begin{itemize}\itemsep=0pt
 \item[$(i)$] Let $\pi_{-k}\colon \mathcal{V}\bigl(* \tilde{D}\bigr) \longrightarrow \mathcal{V}_{\mathfrak{g}_{-k}}\bigl(* \tilde{D}\bigr)$ be the projection on $\mathcal{V}_{\mathfrak{g}_{-k}}$ with respect to $\mathcal{V}_{\mathfrak{g}^{-k+1}}\bigl(* \tilde{D}\bigr)$. Then $\pi_{-k}\circ \mathcal{L} = {\rm Id}_{\mathcal{V}_{\mathfrak{g}_{-k}}}$.
\item[$(ii)$] The restriction of $\mathcal{L}\circ \pi_{-k}$ to $\ker\big(\nabla^\kappa_{\omega_0}\big)$ is the identity on $\ker\big(\nabla^\kappa_{\omega_0}\big)$.
\end{itemize}\end{Lemma}
\begin{proof} The theorem on p.~13 in \cite{Cap4} is exactly the non-singular version of this lemma, i.e., when $D$ is empty. Its proof uses only differential operators constructed with the de Rham differential on trivial modules, and morphisms of modules obtained by tensorizing linear map of complex vector spaces with the identity on holomorphic functions. Thus, it straightforwardly extends to the meromorphic category since such operators preserves the sheaves of meromorphic sections. \end{proof}

\begin{Corollary}\label{proprextparab} Let $(E,\omega_0)$ be a regular holomorphic branched $(G,P)$-Cartan geometry on a~pair $(M,D)$. Suppose that for any irreducible component $D_\alpha$ of $D$, there exist $A\in \mathfrak{g}_-\setminus \mathfrak{g}_{-k}$ and a~$A$-spiral $\Sigma$ of $(E,\omega_0)$ with $\Sigma\cap D_\alpha =\{x_0\}$. Then $(E,\omega_0)$ satisfies the extension property for the infinitesimal automorphisms.
\end{Corollary}
\begin{proof} By Proposition~\ref{branchedtaucondition}, we can suppose that the curves $\Sigma$ in the statement are holomorphic spirals, i.e., admit lifts to $A$-distinguished curves $\tilde{\Sigma}$ with $\tilde{\Sigma}\cap \tilde{D} = \{e_0\}$, for some $e_0 \in p^{-1}(x_0)$ and $A\in \mathfrak{g}_-\setminus \mathfrak{g}_{-k}$.

 Pick an irreducible component $\tilde{D}_\alpha$, let $e_0$ be as above and apply the Theorem~\ref{Lie-Cartan}. Since $k>1$, $\mathbb{C}A$ and $\mathfrak{g}_{-k}$ are independent subspaces in $\mathfrak{g}$. Thus, the projection $\pi_{-k}\big(\ker\big(\nabla^\kappa_{\omega_0}\big)\big)$ extends as a~constant $\mathbb{C}$-subsheaf of $\mathcal{V}_{\mathfrak{g}_{-k}}\bigl(* \tilde{D}\bigr)$ on a neighborhood $U$ of $e_0$. The image of a constant sheaf by a morphism of $\mathbb{C}$-sheaves is a constant sheaf, so by Lemma~\ref{caractersiticparabolic}, $\ker\big(\nabla^{\kappa}_{\omega_0}\big)$ extends as a~constant $\mathbb{C}$-subsheaf of $\mathcal{V}\bigl(* \tilde{D}\bigr)$ on $U$. The proof is then achieved.
\end{proof}

\begin{Remark} Corollary~\ref{proprextparab} admits a meromorphic version as in the degree one case, see Remark~\ref{remarkmeromorphicproprext}. \end{Remark}
\section[Application to the classification of meromorphic affine connections]{Application to the classification of meromorphic affine \\ connections}\label{section4}

In this section, we apply Corollary~\ref{proprextaffine} to the classification of some meromorphic geometric structures. Namely, we give a meromorphic version of Theorem~\ref{theoremBDM} for the \textit{spiral holomorphic branched affine connections}, a subcategory of meromorphic affine connections (see Definition~\ref{branchedaffine}). For, we extend the classical equivalence between affine connections and affine Cartan geometries to the meromorphic setting, and describe the distinguished curves of the corresponding meromorphic Cartan geometries.

\subsection[Equivalence between meromorphic affine connections and meromorphic affine Cartan geometries]{Equivalence between meromorphic affine connections and meromorphic \\ affine Cartan geometries}\label{sectionequivalence}

In this subsection, we consider the complex affine group $G$ of dimension $n\geq 1$, and the complex linear group $P\subset G$. Our aim is to extend the well-known equivalence between affine Cartan geometries and affine connections (see \cite[pp.~362--365]{Sharpe}) to the meromorphic setting.

The restricted adjoint representation $\operatorname{Ad}\colon P \longrightarrow {\rm GL}(\mathfrak{g})$ splits as the sum of two irreducible representations $\mathfrak{g}_-$, the subalgebra corresponding to the infinitesimal generators for the translations in $\operatorname{Aff}(\mathbb{C}^n)$, and $\mathfrak{p}={\rm Lie}(P)$. Consequently, if $E\overset{p}{\longrightarrow} M$ is a holomorphic $P$-principal bundle and $\omega_0$ is a meromorphic $(G,P)$-Cartan connection on $\big(E,\tilde{D}\big)$, then it splits as the sum
\begin{gather}\label{splitCartan} \omega_0 = \theta_0 \oplus \tilde{\omega} \end{gather}
 of a meromorphic solderform~$\theta_0$ on $\big(E,\tilde{D}\big)$ (see Definition~\ref{soldermeromorphe}) and a meromorphic $P$-principal connection~$\tilde{\omega}$ on $\big(E,\tilde{D}\big)$. Conversely, the direct sum of a meromorphic solder form and a meromorphic principal connection on $E$ gives a meromorphic affine Cartan geometry by the formula~\eqref{splitCartan}.

Consider the category $\mathcal{F}_{\rm conn}$ whose objects are triples $\big(\phi_0,\mathcal{E},\overline{\nabla}\big)$ formed by a meromorphic extension $(\phi_0,\mathcal{E})$ over a pair $(M,D)$ and a meromorphic connection $\big(\mathcal{E},\overline{\nabla}\big)$ on $(M,D)$, where an arrow between $(\mathcal{E},\nabla)$ over $(M,D)$ and $(\mathcal{E}',\nabla')$ over $(M',D')$ is defined as follows. It is an isomorphism $(\varphi,\Phi)$ of vector bundle (see \eqref{isovector}) preserving the meromorphic connections, in the sense that $\Phi^* (\varphi^\star \nabla')=\nabla$ where $\varphi^\star \nabla'$ is the pullback of connection (see Definition~\ref{pullback}) and $\Phi^*$ is the pullback in the sheaf-theoretic sense.

Define the map $f$ from the category $\mathcal{G}_{\rm aff}$ of meromorphic $(G,P)$-Cartan geometries on $(M,D)$ to $\mathcal{F}_{\rm conn}$ as follows. If $(E,\omega_0)$ is an object of $\mathcal{G}_{\rm aff}$, consider the meromorphic solderform $(E,\theta_0)$ (see Definition~\ref{soldermeromorphe}) defined by \eqref{splitCartan}, and let $(\mathcal{E},\phi_0)$ be the corresponding meromorphic extension on~$(M,D)$ (see Proposition~\ref{solderequivalence}). Let $\overline{\nabla}$ be the meromorphic connection on $\mathcal{E}=E(\mathbb{C}^n)$ associated with $\tilde{\omega}$ (see~Proposition~\ref{correspondenceprincipal}). We let
 \begin{gather}\label{functorf} f(E,\omega_0)=\big(\phi_0,\mathcal{E},\overline{\nabla}\big).\end{gather}

Now, consider the subcategory $\mathcal{G}_{\rm aff}^0$ of $\mathcal{G}_{\rm aff}$ whose objects are holomorphic branched $(G,P)$-Cartan geometries, together with their isomorphisms. Consider a subcategory $\mathcal{F}_{\rm conn}^0$ of $\mathcal{F}_{\rm conn}$ obtained by intersecting with $\mathcal{F}^0$ (see Definition~\ref{meromorphicextension}).

\begin{Proposition}\label{equivalenceaffineextension} Let $(M,D)$ be a pair. The map $f$ defined by \eqref{functorf} extends to arrows as an equivalence of categories between $\mathcal{G}_{\rm aff}$ \big(resp.\ $\mathcal{G}_{\rm aff}^0$\big) and $\mathcal{F}_{\rm conn}$ \big(resp.\ $\mathcal{F}_{\rm conn}^0$\big).\end{Proposition}
\begin{proof} Let $\Psi\colon E \longrightarrow E'$ be an arrow between two meromorphic $(G,P)$-Cartan geometries $(E,\omega_0)$ and $(E',\omega'_0)$ over $(M,D)$ and $(M',D')$, and $\big(\phi_0,\mathcal{E},\overline{\nabla}\big)$ and $\big(\phi'_0,\mathcal{E}',\overline{\nabla}'\big)$ their images through $f$. So $\Psi$ is a morphism of holomorphic principal bundles between the frame bundles and we define $f(\Psi)=(\varphi,\Phi)$ as the image of $\Psi$ through the equivalence described in Section~\ref{Atiyahframe}. By construction, since $\Psi^\star \omega'_0= \omega_0$, we have $\Psi^\star \theta'_0 = \theta_0$ and $\Psi^\star \tilde{\omega}'=\tilde{\omega}$. The first condition implies that $(\varphi,\Phi)$ is an arrow of meromorphic extensions (see Definition~\ref{meromorphicextension}), while the second one implies that it preserves the meromorphic connections $\nabla$ and $\nabla'$ (see Lemma~\ref{lemmetau}). Hence, $f$ is a functor. Since it is the restriction of the equivalence of categories from Proposition~\ref{solderequivalence}, we obtain an equivalence of categories.
\end{proof}

We can make this equivalence more concrete, though less precise. Let $\big(\phi_0,\mathcal{E},\overline{\nabla}\big)$ be an object of $\mathcal{F}_{\rm conn}$ on $(M,D)$. Then \begin{gather}\label{inducedconnexion} \nabla = \phi_0^{-1} \overline{\nabla} \end{gather}
defines a meromorphic affine connection on $(M,D)$, we will call it the \textit{meromorphic affine connection induced by} $\big(\phi_0,\mathcal{E},\overline{\nabla}\big)$. Thus, there is a functor \begin{gather}\label{fonctmu} \mu\colon\ \mathcal{F}_{\rm conn} \longrightarrow \mathcal{A} \end{gather} to the category $\mathcal{A}$ of meromorphic affine connections on pairs. The composition $\mu\circ f$ is an equivalence of categories between meromorphic affine Cartan geometries and meromorphic affine connections as required.

We denote by $T_\nabla$ the torsion of $\nabla$ \eqref{torsionaffine}. There is the analogous notion of $\mathfrak{g}_-$-\textit{torsion} for an object $(E,\omega_0)$ of $\mathcal{G}_{\rm aff}$ on $(M,D)$. It is the $P$-equivariant meromorphic function~$\tau_{\omega_0}$ on~$E$ with values in $\mathbb{W}_{\mathfrak{g}_-} = \bigwedge^{2} (\mathfrak{g}_-)^* \otimes \mathfrak{g}_-$ defined as the projection of the Cartan curvature~$k_{\omega_0}$ of~$(E,\omega_0)$ (see Definition~\ref{curvature}) on $\mathbb{W}_{\mathfrak{g}_-}$ respective to $\bigwedge^{2} (\mathfrak{g}_-)^* \otimes \mathfrak{p}$.

 Finally, let us remark that for any object $\big(\mathcal{E},\phi_0,\overline{\nabla}\big)$ of $\mathcal{F}_{\rm conn}^0$, the meromorphic affine connection~\eqref{inducedconnexion} restricts as a holomorphic connection on the submodule $\mathcal{E}$. We then define: \begin{Definition}\label{branchedaffine} The category $\mathcal{A}^0$ is the subcategory of $\mathcal{A}$ whose objects are the meromorphic affine connections on $(M,D)$ preserving a locally free $\mathcal{O}_M$-module $\mathcal{E}$ with $TM\subset \mathcal{E}\subset TM(* D)$, in the above sense. Its objects 	are called holomorphic branched affine connections. \end{Definition}

 The following lemma is the only missing piece to restrict the equivalence to the branched affine subcategories.

 \begin{Lemma}\label{lemmabranchedaffine} Let $\nabla$ be a holomorphic branched affine connection on $(M,D)$. Then the submodule $\mathcal{E}\subset TM(* D)$ from Definition~{\rm \ref{branchedaffine}} is unique. \end{Lemma}
 \begin{proof} Let $E$ be the bundle of holomorphic frames for $\mathcal{E}$, and $\tilde{\omega}$ be the meromorphic principal connection on $R^1(M)$ corresponding to $\nabla$ (see Proposition~\ref{correspondenceprincipal}). Suppose there exists another rank~$n$ locally free submodule $\mathcal{E}'$ of $TM(* D)$ such that $\nabla$ restricts to a holomorphic connection on~$\mathcal{E}'$, and let $\tilde{\omega}'$ be the corresponding holomorphic principal connection on its bundle of holomorphic frames $E'$.

 Pick a point $x\in M$, and a neighborhood $U$ of $x$ in $M$ with two basis $(\overline{s}_1,\dots,\overline{s}_n)$ of~$\mathcal{E}\vert_U$ and $\big(\overline{t}_1,\dots,\overline{t}_n\big)$ of $\mathcal{E}'\vert_U$. Denote by $\sigma$, $\sigma'$ the holomorphic sections of $R^1(M\setminus D)$ on $U\setminus D$ corresponding respectively to these basis, and $b$ be the unique holomorphic function on $U\setminus D$ with values in ${\rm GL}_n(\mathbb{C})$ such that $\sigma'=\sigma\cdot b$. The classical gauge formula implies that $b$ must be a solution of the differential equation
 \begin{gather*} {\rm d}(b) = Ab - bA', \end{gather*}
 where $A$ (resp.\ $A'$) is the matrix of $\nabla$ in the basis $\sigma$ (resp.\ $\sigma'$). Since $A$ and $A'$ are holomorphic on $U$, we can use the proof of the Proposition II.2.13 in \cite{Sabbah} to obtain that $b$ extends on $U$ as a~holomorphic function. Reversing the roles of $\sigma$ and $\sigma'$, this is also true for $b^{-1}$, so that $\mathcal{E}$ and~$\mathcal{E}'$ coincide over $U$. We get the unicity.
 \end{proof}

\begin{Corollary}\label{equivalencetorsion} The composition of the equivalence from Proposition~{\rm \ref{equivalenceaffineextension}} and the map given by~\eqref{fonctmu} gives an equivalence of categories between $\mathcal{G}_{\rm aff}$ \big(resp.\ $\mathcal{G}_{\rm aff}^0$\big) and $\mathcal{A}$ \big(resp.\ $\mathcal{A}^0\big)$.
\end{Corollary}
Recall that, given any meromorphic affine connection $\nabla$ on $(M,D)$, a local holomorphic vector field $\overline{X}$ on an open subset $U\subset M$ is a \textit{Killing field} for $\nabla$ iff the pullback of $\nabla$ by its flows is again $\nabla$. We denote by $\mathfrak{k}\mathfrak{i}\mathfrak{l}\mathfrak{l}_{\nabla}$ the subsheaf of $TM\setminus D$ whose sections are the Killing field for~$\nabla$. By Corollary~\ref{equivalencetorsion}, we obtain the following lemma.

\begin{Lemma}If $(E,\omega_0)$ is an object of $\mathcal{G}_{\rm aff}$ on a pair $(M,D)$, and $\nabla$ the corresponding meromorphic affine connection on $(M,D)$, then $\mathfrak{k}\mathfrak{i}\mathfrak{l}\mathfrak{l}_{M,\omega_0}=\mathfrak{k}\mathfrak{i}\mathfrak{l}\mathfrak{l}_{\nabla}$. \end{Lemma}

\subsection{Geodesics of spiral holomorphic branched affine connections}

Let $\nabla$ be a holomorphic affine connection on a complex manifold $M$, and $(E,\omega)$ be a holomorphic affine Cartan geometry inducing it. Recall that \textit{geodesics} $\Sigma\subset M$ of $\nabla$ are the projections of the $A$-distinguished curve of $(E,\omega_0)$ for $A\in \mathfrak{g}_-$. Equivalently, these are the images of holomorphic parametrized curves $\gamma\colon D(0,\epsilon) \longrightarrow M$ such that
\begin{gather}\label{affinegeodesic} \gamma^\star \nabla \left({\rm d}\gamma\left(\der{}{t}\right)\right) = 0, \end{gather}
where ${\rm d}\gamma\big(\der{}{t}\big)$ is the image of the canonical vector field on the open disk $D(0,\epsilon) \subset \mathbb{C}$ through the differential ${\rm d}\gamma \colon TC \longrightarrow \mathcal{O}_C\otimes \gamma^{-1}TM$, and $\gamma^\star \nabla$ is the pullback (see Definition~\ref{pullback}).
\begin{Definition}\label{meromorphicgeodesics}Let $\nabla$ be a meromorphic affine connection on a pair $(M,D)$. A \textit{geodesic} of $\nabla$ is a curve $\Sigma \subset M$ whose restriction to $M\setminus D$ is locally the image of a geodesic of the holomorphic affine connection $\nabla\vert_{M\setminus D}$ in the above sense.
 \end{Definition}

Note that such a definition permits any holomorphic curve $\gamma$ with image contained in $D$ to be a geodesic for $\nabla$. This is coherent with the classical definition \eqref{affinegeodesic} since the pullback $\gamma^\star \nabla$ will often be the zero morphism for such a curve (see Example~\ref{exemplegeodesicisotrope}).

\begin{Lemma}\label{caracterisationbranchedgeodesics} Let $\nabla$ be a holomorphic branched affine connection on a pair $(M,D)$. Let $\Sigma \subset M$ be a curve. Then the following assertions are equivalent:
 \begin{itemize}\itemsep=0pt
 \item[$(i)$] $\Sigma$ is a geodesic for $\nabla$.
\item[$(ii)$] For any non-constant holomorphic curve $\gamma\colon D(0,\epsilon) \longrightarrow \Sigma$, $\gamma^\star \nabla$ is the null morphism or there exists a holomorphic function $h_\gamma$ on $D(0,\epsilon)$, which does not identically vanish, and such that
\[\gamma^\star \nabla\left(\frac{1}{h_\gamma} \der{}{t}\right) =0,\]
 where $\der{}{t}$ and $\gamma^\star \nabla$ are defined as above.
 \end{itemize}\end{Lemma}

\begin{proof} Let $(E,\omega_0)$ be the unique holomorphic branched affine Cartan geometry inducing $\nabla$ through the equivalence of Proposition~\ref{equivalenceaffineextension}, and such that $\mathcal{E}=E(\mathbb{C}^n)$ is the submodule of Definition~\ref{branchedaffine}. By Lemma~\ref{branchedfoliation}, the $\mathfrak{p}$-component of $\omega_0$ defines a holomorphic $P$-principal connection on $E$. We denote by $\theta_0$ the $\mathfrak{g}_-$-component of $\omega_0$.

 We first prove that $(i)$ implies $(ii)$. By the above remark, for any choice of a point $e_0$ in the fiber of $\gamma(0)$, there exists an unique lift $\tilde{\gamma}$ of $\gamma$ to a holomorphic curve tangent to $\ker(\tilde{\omega})$ and satisfying $\tilde{\gamma}(0)=e_0$. Suppose that $\gamma^\star \nabla$ is not the null morphism. By definition of $\nabla$, the hypothesis~$(i)$ implies that for any $t_0 \in D(0,\epsilon)$ such that $\gamma(t_0)$ does not lie in $D$, there exists a~holomorphic function $h$ in a neighborhood of $t_0$ with $\tilde{\gamma}^\star \theta_0 \big(h(t)\der{}{t}\big)$ constant, that is there exists $A\in \mathfrak{g}_-$ with $\tilde{\gamma}^\star \theta_0\big(\der{}{t}\big) = \frac{1}{h(t)} A$ in a neighborhood of $t_0$. It is clear that the direction~$\mathbb{C}A$ is independent from the point $t_0$ as before. Now, since $\omega_0$ is a holomorphic branched Cartan connection, $\tilde{\gamma}^\star \theta_0$ is a~holomorphic one-form on $D(0,\epsilon)$ with values in $\mathbb{C}A$. In particular, $h_\gamma= \tilde{\gamma}^\star \theta_0\big(\der{}{t}\big)$ is a holomorphic function on $D(0,\epsilon)$, which is not identically vanishing by the assumption. By definition of~$\gamma^\star \nabla$ this implies $\gamma^\star \nabla\big(\frac{1}{h_\gamma}\der{}{t}\big)=0$. We have proved $(ii)$.

Now we prove $(ii)$ implies $(i)$. If $\gamma^\star \nabla$ is the null morphism, then by the Leibniz identity we must have $\mathcal{O}_\mathbb{C} \otimes \gamma^{-1} \mathcal{O}_{M}(D) = \underline{\{0\}}_{\mathbb{C}}$. Thus, the image of $\gamma$ must lie in $D$, and $\gamma$ is a geodesic. Suppose $\gamma^\star \nabla$ is not the null morphism. Then we can use the above description of $\gamma^\star \nabla\big(\der{}{t}\big)$ to conclude that $\tilde{\gamma}^\star \theta_0\big(\der{}{t}\big)$ is a holomorphic function with values in $\mathbb{C}A$ for some $A\in \mathfrak{g}_-$. Since $\tilde{\gamma}$ is tangent to $\ker(\tilde{\omega})$, we recover that its image $\tilde{\Sigma}$ is a $A$-distinguished curve for $(E,\omega_0)$ projecting onto the image of $\gamma$, that is $(i)$ is satisfied.
\end{proof}

\begin{Definition}\label{tauconnection} A meromorphic affine connection $\nabla$ on a pair $(M,D)$ is said to be a \textit{spiral} connection if no irreducible component of $D$ is invariant by the geodesics of $\nabla$ in the sense of Definition~\ref{meromorphicgeodesics}. \end{Definition}

The Lemma~\ref{caracterisationbranchedgeodesics} and the Proposition~\ref{branchedtaucondition} then admit the following consequence.

\begin{Corollary}\label{tauconnectionstaucondition} Let $\nabla$ be a holomorphic branched holomorphic $\nabla$ on a pair $(M,D)$. If $\nabla$ is a spiral connection, then any holomorphic branched affine Cartan geometry $(E,\omega_0)$ inducing $\nabla$ through the equivalence of Corollary~{\rm \ref{equivalencetorsion}} is strongly spiral $($see Definition~{\rm \ref{strongtaucondition})}. \end{Corollary}
\begin{proof} In virtue of Lemma~\ref{lemmabranchedaffine}, $(E,\omega_0)$ is isomorphic to the holomorphic branched Cartan geometry where $E=R(\mathcal{E})$ is the frame bundle of the unique submodule $\mathcal{E}\subset TM(* D)$ on which $\nabla$ restricts as a holomorphic connection, and $\omega_0=\theta_0 \oplus \tilde{\omega}$ where $\tilde{\omega}$ is the corresponding $P$-principal connection and $\theta_0$ the meromorphic solderform of $E$ (see Definition~\ref{soldermeromorphe}).

The proof of Lemma~\ref{caracterisationbranchedgeodesics} recall the basic fact that the restrictions of geodesics for $\nabla$ to $M\setminus D$ are exactly the projections through $p\colon E\longrightarrow M$ of the $A$-distinguished curves of $(E|_{M\setminus D},\omega_0)$, for some $A\in \mathfrak{g}_-$. These are in particular $A$-spirals for $(E,\omega_0)$. Thus, using Proposition~\ref{branchedtaucondition}, we get the desired implication. \end{proof}

\subsection[Spiral holomorphic branched affine connections in algebraic dimension zero]{Spiral holomorphic branched affine connections \\ in algebraic dimension zero}

Now, we will give an application of the results of Section~\ref{killparab} to the classification of affine meromorphic connections on some simply connected complex compact manifolds $M$. Most of them are adaptations of the arguments used in the proof of the principal theorem in \cite{BDM}, using the results of this article.

\begin{Theorem}\label{quasihomogeneousaffine} Let $(M,D)$ be a pair, with $M$ a simply connected complex compact manifold. If~$(M,D)$ bears a spiral quasihomogeneous meromorphic affine connection $\nabla$ $($see Definition~{\rm \ref{tauconnection})}, then it admits a meromorphic parallelism $\big(\overline{X}_1,\dots,\overline{X}_n\big)$, such that $\overline{X}_i$ is a Killing vector field for $\nabla$ when restricted to $M\setminus D$.

\end{Theorem}
\begin{proof}
 Let $(E,\omega_0)$ be the meromorphic affine Cartan geometry on $M$ corresponding to $\nabla$ (see~Corollary~\ref{equivalencetorsion}). By the Corollary~\ref{proprextaffine}, it satisfies the extension property of infinitesimal automorphisms, i.e., the local system $\mathfrak{k}\mathfrak{i}\mathfrak{l}\mathfrak{l}_{\nabla}$ on $M\setminus D$ extends as a local system $\mathcal{K}$ on $M$, with $\mathcal{K}\subset TM(* D)$. Since $M$ is simply connected, this is a constant sheaf on $M$. Since $\nabla$ is assumed quasihomogeneous, we can pick $x\in M$ and a $\mathcal{O}_{M,x}$-basis $\overline{X}_{1,x},\dots,\overline{X}_{n,x}$ of $(TM)_x$ formed by germs of Killing fields for $\nabla$. These germs are thus restrictions of global meromorphic vector fields $\overline{X}_1,\dots,\overline{X}_n$ whose restrictions to $M\setminus D$ are elements of $\mathfrak{k}\mathfrak{i}\mathfrak{l}\mathfrak{l}_\nabla(M\setminus D)$. Since their germs at $x$ are linearly independent, there exists a Zariski-dense open subset $M\setminus S$ such that the restrictions of $\overline{X}_1,\dots,\overline{X}_n$ to any subset $U\subset M\setminus S$ are independent elements of $TM(U)$, i.e., $\big(\overline{X}_1,\dots,\overline{X}_n\big)$ is a meromorphic parallelism on $M$.
\end{proof}

We obtain the following theorem.

\begin{Theorem}\label{classificationaffine} Let $M$ be a compact complex manifold with finite fundamental group, and whose meromorphic functions are constants. Then $M$ does not bear any spiral branched holomorphic affine connection.
\end{Theorem}
\begin{proof} Suppose that $\nabla$ is a spiral branched holomorphic affine connection on $(M,D)$, denote by $\mathcal{E}$ the submodule of $TM(* D)$ from Lemma~\ref{lemmabranchedaffine}. Denote also $E$ the frame bundle of $\mathcal{E}$, so that in particular $\mathcal{E}=E(\mathfrak{g})/E(\mathfrak{p})$. Then by Proposition~\ref{equivalenceaffineextension}, there exists a unique branched holomorphic Cartan geometry of the form $(E,\omega_0)$ inducing $(\mathcal{E},\nabla)$.

Complete the meromorphic parallelism $\big(\overline{X}_i\big)_{i=1,\dots,n}$ as in the proof of Theorem~\ref{quasihomogeneousaffine} into a~basis $\big(\overline{X}_j\big)_{j=1,\dots,r}$ for the global meromorphic Killing fields of $\nabla$. A meromorphic parallelism is a~rigid geometric structure (see \cite{DumitrescuQuasi}), so by~\cite[Theorem 2]{Dumitrescu0}, the juxtaposition of~$\big(\overline{X}_j\big)_{j=1,\dots,r}$ and $\nabla$ is quasihomogeneous. Since $\nabla$ satisfies the extension property for the Killing vector fields~(see Corollary~\ref{proprextaffine}) and $M$ is simply connected, we obtain a meromorphic parallelism $\big(\overline{X}'_i\big)_{i=1,\dots,n}$ such that the restriction of each $\overline{X}'_i$ to $M\setminus D$ is a Killing field for $\nabla$ and commutes with each $\overline{X}_j$. In particular, each $\overline{X}'_i$ is a $\mathbb{C}$-linear combination of the $\big(\overline{X}_j\big)_{j=1,\dots,r}$, so $\big(\overline{X}'_i\big)_{i=1,\dots,n}$ are commuting meromorphic vector fields.

Now, let pick any Gauduchon metric on $M$ (called a \textit{standard metric} in \cite[p.~502]{Gauduchon}) and let us prove that the degree $\deg (\mathcal{E})$ of $\mathcal{E}$ with respect to this metric is zero. Let $(E,\omega_0)$ be the branched holomorphic affine Cartan geometry on $(M,D)$ corresponding to $\nabla$ (see Corollary~\ref{equivalencetorsion}). Then $\mathcal{E}=E(\mathfrak{g}/\mathfrak{p})=E(\mathfrak{g})/E(\mathfrak{p})$ by definition of $(E,\omega_0)$, and since $P={\rm GL}_n(\mathbb{C})$, we have $\deg (E(\mathfrak{p}))=0$ (see \cite[Corollary 4.2]{BD}).

We must then prove that $\deg (E(\mathfrak{g}))=0$. For, it is sufficient to prove that $C_1(R_{\nabla^{\omega_0}})$ vanishes identically, where $C_1$ is the trace on $\operatorname{End}(E(\mathfrak{g}))$ and $\nabla^{\omega_0}$ is the tractor connection (see Definition~\ref{tractorconnection}). We will prove that the meromorphic one-form $\eta_i = X'_i \lrcorner C_1(R_{\nabla^{\omega_0}})$ vanishes identically on $M$ for any $1\leq i\leq n$. By Lemma~\ref{curvaturetractor}, we have
 \begin{gather*}
 p^\star \eta_i = \tilde{X}'_i \lrcorner C_1(R_{p^\star \nabla^{\omega_0}}) = \underset{\tilde{\eta}_i^0}{\underbrace{{\rm d}C_1(\operatorname{ad}(s_i))}} + \underset{\tilde{\eta}^1_i}{\underbrace{\tilde{X}'_i \lrcorner \operatorname{ad}(\omega_0 \wedge \omega_0)}},\end{gather*}
 where $\tilde{X}'_i$ is the lifting of $X'_i$ to $E$ and $s_i = \omega_0(X'_i)$.

The meromorphic one-form $\tilde{\eta}^0_i$ is exact and $P$-equivariant. By a classical result on exact invariant forms on connected Lie groups, the restriction of $\tilde{\eta}^0_i$ to any fiber of $E\overset{p}{\longrightarrow} M$ corresponds to a homomorphism $\chi\colon P \longrightarrow \mathbb{C}$. Because $P={\rm GL}_n(\mathbb{C})$, any such homomorphism is trivial, so that $\tilde{\eta}^0_i$ vanishes on the fibers of $E\overset{p}{\longrightarrow} M$. Thus, it is the pullback of a meromorphic exact one-form $\eta^0_i$ on $M$. Moreover, by Corollary~\ref{proprextaffine}, $s_i$ is a holomorphic section of $\mathcal{V}$ on $E$, so that $\tilde{\eta}^0_i$ is a holomorphic one-form. Thus, $\eta^0_i$ is an exact holomorphic one-form on a simply connected compact complex manifold, i.e., vanishes everywhere.

Now, let us prove that $\tilde{\eta}^1_i = p^\star \eta_i$ vanishes everywhere. Consider $E_G=E\times G \overset{pi_G}{\longrightarrow}E$ and the holomorphic tractor connection $\tilde{\omega}$ on it (see Definition~\ref{tractorconnection}). Using the splitting $TE_G = \ker(\tilde{\omega}) \oplus \ker(d\pi_G)$, the pullback $\hat{\eta}^1_i= \pi_G^\star \tilde{\eta}^1_i$ uniquely decomposes as a sum
\[\hat{\eta}^1_i = \hat{\eta}_i^H \oplus \hat{\eta}_i \]
 with $\hat{\eta}_i^H$ a $G$-invariant meromorphic one-form on $E_G$, vanishing on $\ker(\tilde{\omega})$, and $\hat{\eta}^V_i$ vanishing on~$\ker(\tilde{\omega})$. In particular, $\hat{\eta}^H_i$ is the pullback of $\eta_i$ through the composition $p_G=p\circ \pi_G$, so that~$\hat{\eta}^V_i$ vanishes everywhere. Now, using Corollary~\ref{proprextaffine}, $\tilde{\eta}^1_i$ is a~holomorphic one-form, so that~$\eta_i$ is a~holomorphic one-form on $M$. Using the Lie--Cartan formula, we have
\[ {\rm d}\eta_i(X'_j,X'_k)=\mathcal{L}_{X'_j}\eta_i(X'_k) - \mathcal{L}_{X'_k}\eta_i(X'_j) - \eta_i([X'_j,X'_k]_{TM}). \]
 Since the only meromorphic functions on $M$ are the constants, we obtain \[\mathcal{L}_{X'_j}\eta_i(X'_k) =\mathcal{L}_{X'_k}\eta_i(X'_j) =0,\] and since the meromorphic vector fields $(X'_i)_{i=1,\dots,n}$ commute, $\eta_i$ is a closed holomorphic one-form. Since $M$ is simply connected and compact, $\eta_i$ vanishes everywhere. This proves that $C_1(R_{\nabla^{\omega_0}})$ vanishes everywhere, i.e., $\deg (E(\mathfrak{g}))=0$.

Hence, $\deg (\mathcal{E})=0$. Let $\overline{s}_1,\dots,\overline{s}_n$ be the images of $X'_1,\dots,X'_n$ through the morphism $\phi_0$, where $(\phi_0,\mathcal{E})$ is the holomorphic extension image of $(E,\omega_0)$ as in Proposition~\ref{equivalenceaffineextension}. Since $s_i=\omega_0(\tilde{X}'_i)$ is a section of $\mathcal{V}(E)$ for any $1\leq i \leq n$ from Corollary~\ref{proprextaffine}, each $\overline{s}_i$ is a section of~$\mathcal{E}(M)$. Since they are independent, the holomorphic section $\bigwedge_{i=1}^{n} \overline{s}_i$ of $\det(\mathcal{E})$ is not identically vanishing, thus $\det(\mathcal{E})$ is trivial and $\bigwedge_{i=1}^{n}\overline{s}_i$ never vanishes. It therefore forms a basis of $\mathcal{E}$ on $M$, and the dual sections $\overline{s}_1^*,\dots,\overline{s}_n^*$ are holomorphic sections of $\mathcal{E}^*$ on $M$. We obtain a meromorphic $(\mathbb{C}^n,\{0\})$-Cartan geometry $(M,\eta)$ on $(M,D)$ given by the meromorphic one-form
\[ \eta \,=\, \bigwedge_{i=1}^{n} \,(\overline{s}_i^* \circ \phi_0)\, \otimes \,\mathfrak{e}_i,\]
 where $(\mathfrak{e}_i)_{i=1,\dots,n}$ is the canonical basis of $\mathbb{C}^n$. Since $\omega_0$ is a holomorphic one-form on $E$, by definition of the associated meromorphic extension, $\phi_0$ restricts to an injective morphism of $\mathcal{O}_M$-modules from $TM$ to $\mathcal{E}$. Hence $\eta$ is a holomorphic one-form on $M$, i.e., $(M,\eta)$ is a branched holomorphic Cartan geometry. Because the $\eta$-constant vector fields $(X'_i)_{i=1,\dots,n}$ commute, it is in fact a flat branched holomorphic Cartan geometry. Since $M$ is simply connected and compact, there is a holomorphic submersion $\operatorname{dev}\colon M \longrightarrow \mathbb{C}^n$. This is impossible by the maximum principle. So $M$ cannot bear any spiral branched holomorphic affine connection.
\end{proof}

\section{Genericity of the spiral property on surfaces}
In this section, we fix a pair $(M,D)$ where $M$ is a complex surfaces and $D$ an effective divisor of~$M$ with irreducible and reduced components $(D_\alpha)_{\alpha\in I}$. We fix a submodule $\mathcal{E} \subset TM(* D)$ with $TM\subset \mathcal{E}$.
\begin{Definition}\label{A0} In the above setting, we denote by $\mathcal{A}^0_{\mathcal{E}}$ the set of holomorphic branched affine connections on $(M,D)$ whose associated submodule is $\mathcal{E}$ (see Definition~\ref{branchedaffine}). We denote by $\mathcal{A}^0_{\mathcal{E},\tau}$ the subset of~$\mathcal{A}^0_{\mathcal{E}}$ consisting of spiral meromorphic affine connection (see Definition~\ref{tauconnection}). \end{Definition} We prove a result of genericity for $\mathcal{A}^0_{\mathcal{E},\tau}$ (see Theorem~\ref{genericity}). We also give examples of flat and spiral holomorphic branched affine connections on compact complex manifolds of arbitrary dimension, and one example of a non-flat and spiral holomorphic branched affine connection on a complex compact threefold.

\subsection[Consequence of the existence of a spiral affine connection on the submodule]{Consequence of the existence of a spiral affine connection\\ on the submodule}
We begin with a necessary condition on $\mathcal{E}$ for $\mathcal{A}^0_{\mathcal{E},\tau}$ not being empty.
\begin{Proposition}\label{necessarynottau} Suppose the existence of a spiral branched holomorphic affine connection $\nabla$ in $\mathcal{A}^0_{\mathcal{E},\tau}$. Then $\mathcal{E}$ satisfies the following property. Let $D_\alpha$ be an irreducible component of $D$, and~$x\in D_\alpha$. Let $(z_1,z_2)$ be local coordinates on an open neighborhood $U$ of $x$ with
\[D_\alpha \cap U=\{z_1=0\}.\]
The matrix of any element $\varphi \in \operatorname{End}(\mathcal{E})(U)$, identified with an element of $\operatorname{End}(\mathcal{E})(* D)(U)$, in $\big(\der{}{z_1},\der{}{z_2}\big)$ takes the form
\begin{equation}\label{formnottau} \begin{pmatrix} * & z_1 f_1 \\ * & f_2 \end{pmatrix}, \end{equation} where $f_1$, $f_2$ are holomorphic functions on $U$.\end{Proposition}

\begin{proof} Let $\nabla$, $D_\alpha$, $x$ and $U$ as in the statement. Let $(E,\omega_0)$ be the unique holomorphic branched affine Cartan geometry inducing $\nabla$ and such that $E$ is the frame bundle of $\mathcal{E}$. The hypothesis on $\nabla$ implies that the pullback $\tilde{D}_\alpha$ of $D_\alpha$ to $E$ is invariant through the $A$-distinguished foliation $\mathcal{T}_A$ (see \eqref{definitionTA}), for any $A\in \mathfrak{g}_-$. Using the proof of Lemma~\ref{branchedfoliation} and denoting by $\big(\tilde{\der{}{z_1}},\tilde{\der{}{z_2}}\big)$ two $P$-invariant holomorphic vector fields on $p^{-1}(U)$ projecting on $\big(\der{}{z_1},\der{}{z_2}\big)$, we get that the $\omega_0$-constant vector fields $Y_1$, $Y_2$ associated with $\mathfrak{e}_1$, $\mathfrak{e}_2$ decompose as
\[ Y_i = z_1^{n_i}g_i \tilde{\der{}{z_1}} + z_1^{m_i}g'_i \tilde{\der{}{z_2}} + Y'_i,\]
 where $Y'_i$ is an element of $\ker({\rm d}p)\big(p^{-1}(U)\big)$, $g_i$, $g'_i$ are invertible or identically vanishing holomorphic functions, and
 \begin{gather}\label{inequalityni}
 0\geq n_i>m_i \qquad \text{or} \qquad g_i =0.\end{gather}
 Now, fix any holomorphic section $\sigma$ of $E$ on $U$ and denote by $\big(\overline{Y}_1,\overline{Y}_2\big)$ the corresponding basis of $\mathcal{E}$ on $U$. These are the projections of $Y_1\circ \sigma$ and $Y_2\circ \sigma$ through $Tp$, so that
\[\overline{Y}_i = z_1^{n_i}g_i \der{}{z_1} + z_1^{m_i}g'_i \der{}{z_2}.\]
 The inequality \eqref{inequalityni} implies that, up to replacing $\sigma$ by $\sigma\cdot b$ for some holomorphic function $b\colon U \longrightarrow P$, we can suppose that $g_1=0$, i.e., the matrix $Q$ of $\big(\overline{Y}_1,\overline{Y}_2\big)$ in $\big(\der{}{z_1},\der{}{z_2}\big)$ is
 \begin{gather*}Q= \begin{pmatrix} 0 & z_1^{n_2}g_2 \\ z_1^{m_1}g'_1 & z_1^{m_2}g'_2 \end{pmatrix} \end{gather*}
 with $g_2$, $g'_1$, $g'_2$, $n_2$ and $m_2$ as before. Since $\mathcal{E}\vert_U$ contains $TU$, the inverse of this matrix must be a holomorphic matrix on $U$.

 This implies the following identities:
 \begin{align}
 n_2+m_2-m_1 \geq 0,\qquad n_2 <0,\qquad m_1 < 0, \qquad
 m_2 \leq 0.\label{conditionsordres}
\end{align}

 Let $\varphi \in \operatorname{End}(\mathcal{E})(U)$ with matrix $\left(\begin{smallmatrix} \alpha & \gamma \\ \beta & \delta \end{smallmatrix}\right)$ in $\big(\overline{Y}_1,\overline{Y}_2\big)$. Then the matrix of the same section in the basis $\big(\der{}{z_1},\der{}{z_2}\big)$ is
 \begin{gather*} Q \begin{pmatrix} \alpha & \gamma \\ \beta & \delta \end{pmatrix} Q^{-1} = \begin{pmatrix} * & g z_1^{n_2-m_2}\\ * & g' z_1^{m_2-m_1} \end{pmatrix} \end{gather*}
 for some holomorphic functions $g$, $g'$ on $U$. Using \eqref{conditionsordres}, we get that this matrix has the desired form.
\end{proof}

\subsection{Intermediate condition and local characterization}
We continue by introducing a subset $\mathcal{A}^0_{\mathcal{E},1}\subset \mathcal{A}_\mathcal{E}^0$, containing the complement of $\mathcal{A}^0_{\mathcal{E},\tau}$ in $\mathcal{A}^0_{\mathcal{E}}$. This subset has the advantage that its elements $\nabla$ can be described by their local Christoffel symbols.

\begin{Definition}\label{intermediate} The set $\mathcal{A}^0_{\mathcal{E},1}$ is the subset of elements $\nabla \in \mathcal{A}^0_\mathcal{E}$, with the following property. For any $x\in D_\alpha$, where $D_\alpha$ is some irreducible component of $D$, there exists a non-constant geodesic $\gamma\colon D(0,\epsilon) \longrightarrow M$ for $\nabla$ with
\[ \gamma(0)=x \qquad \text{and} \qquad \operatorname{Im}(\gamma) \subset D_\alpha.\]
\end{Definition}

\begin{Lemma}\label{caracterisationintermediate} Let $\nabla \in \mathcal{A}^0_\mathcal{E}$. The following properties are equivalent:
 \begin{itemize}\itemsep=0pt
 \item[$(i)$] $\nabla \in \mathcal{A}^0_{\mathcal{E},1}$.
 \item[$(ii)$] For any irreducible component $D_\alpha$ of $D$, any $x\in D_\alpha \setminus \bigcup_{\beta \neq \alpha} D_\beta$, and any open neighborhood~$U$ of $x$ with local coordinates $(z_1,z_2)$ as in Proposition~{\rm \ref{necessarynottau}}, the matrix of $\nabla$ in $\big(\der{}{z_1},\der{}{z_2}\big)$ is of the form
 \[\sum_{i=1,2}{\rm d}z_i \otimes \begin{pmatrix} a_i & c_i \\ b_i & d_i \end{pmatrix} \]
 with
 \begin{align*}
 \begin{cases} c_2 \quad\text{vanishing on $D_\alpha\cap U$,} \\ d_2 \quad \text{holomorphic.} \end{cases}\end{align*}
 \end{itemize}

\end{Lemma}
\begin{proof}
We first prove $(i)$ implies $(ii)$. Suppose that $\nabla$ is as in $(i)$, and let $\gamma$ be a non constant geodesic at $x\in D_\alpha$ with $\operatorname{Im}(\gamma)\subset D_\alpha$. Let $U$ and $(z_1,z_2)$ as in $(ii)$ and denote by $\gamma_i = z_i \circ \gamma$, so that $\gamma'_1$ is identically vanishing on $D(0,\epsilon)$ by the hypothesis on $\gamma$. The first line of the generalized geodesic equation from Lemma~\ref{caracterisationbranchedgeodesics} implies that $c_2 \circ \gamma$ must be identically vanishing, since~$\gamma'_2$ is not identically vanishing by hypothesis. In particular, since $D_\alpha \cap U$ is irreducible, $c_2$ must be identically vanishing on this divisor.
Moreover, the second line of the same equation implies that $y=\frac{1}{h_\gamma} \gamma'_2$ is a solution of the differential equation
 \begin{gather}\label{residualgeodesic} y'= ({\rm d}_2\circ \gamma) y. \end{gather}
 Now, let $\tilde{\gamma}$ be the lifting of $\gamma$ to a $A$-distinguished curve of the unique holomorphic branched affine Cartan connection $(E,\omega_0)$ inducing $\nabla$ and with $E=R(\mathcal{E})$. Then $\tilde{\gamma}^\star \omega_0$ is of constant rank. Recall that the holomorphic function $h_\gamma$ in Lemma~\ref{caracterisationbranchedgeodesics} is defined by $\tilde{\gamma}^\star \omega_0\big(\der{}{t}\big) = h_\gamma A$. But $\omega_0$ has constant rank on the pullback $\tilde{D}_\alpha$ of $D_\alpha$, by definition of $D$. Thus, $h_\gamma$ must be an invertible holomorphic function on $D(0,\epsilon)$, so that $y$ is a holomorphic function. Then ${\rm d}_2\circ \gamma$ must be holomorphic, and since $\gamma'(0)\neq 0$, this implies that ${\rm d}_2$ has no pole along $D_\alpha \cap U$. But the poles of $\nabla$ are contained in $D$, so that ${\rm d}_2$ is a holomorphic function on the whole~$U$.

Now we prove $(ii)$ implies $(i)$. Let $\nabla \in \mathcal{A}^0_\mathcal{E}$ and suppose $(ii)$ is satisfied. Let $x \in D_\alpha$, and $(U,(z_1,z_2))$ as above, and suppose moreover that $z_1(x)=z_2(x)=0$. Define $\gamma_1 =0$, and $\gamma_2$ to be a solution of \eqref{residualgeodesic} on $D(0,\epsilon)$ ($\epsilon >0$) with $\gamma_2(0)=0$. Such a solution exists because $d_2$ is holomorphic. Then the unique holomorphic curve $\gamma\colon D(0,\epsilon)$ with $\gamma_i = z_i \circ \gamma$ is a geodesic of $\nabla$ by Lemma~\ref{caracterisationbranchedgeodesics}. By construction, $\operatorname{Im}(\gamma)\subset D_\alpha$, i.e., $\nabla\in \mathcal{A}^0_{\mathcal{E},1}$.
\end{proof}
\subsection{Genericity result}
The set $\mathcal{A}^0_\mathcal{E}$ has the structure of an affine space directed by $\operatorname{End}(\mathcal{E})(M)$. Indeed, suppose there exists $\nabla \in \mathcal{A}^0_\mathcal{E}$. Let $\nabla'$ be any meromorphic connection on $(M,D)$, and
 \begin{gather}\label{parametrizedbranched} \Theta = \nabla' - \nabla, \end{gather}
 which is an element of $\Omega^1_M\otimes \operatorname{End}(TM)(* D)(M) = \Omega^1_M\otimes \operatorname{End}(\mathcal{E})(* D)$. Then it is immediate that $\nabla' \in \mathcal{A}^0_\mathcal{E}$ exactly when $\Theta\in \Omega^1_M\otimes \operatorname{End}(\mathcal{E})(M)$.

Using the above remark and Lemma~\ref{caracterisationintermediate}, we get the following result.

\begin{Theorem}\label{genericity} Let $(M,D)$ be a pair and $\mathcal{E}\subset TM(* D)$ a submodule containing $TM$. Then one and only one of the following assertions is true: \begin{itemize}\itemsep=0pt
 \item[$(a)$] $\mathcal{A}^0_{\mathcal{E},\tau}=\mathcal{A}^0_\mathcal{E}$,
\item[$(b)$] $\mathcal{A}^0_{\mathcal{E},1}=\mathcal{A}^0_\mathcal{E}$.
\end{itemize} \end{Theorem}
\begin{proof}
Suppose $(a)$ to be false and let us prove $(b)$. By hypothesis, there exists an element $\nabla\in \mathcal{A}^0_\mathcal{E}$ which is not spiral. Fix any element $\nabla'$ of $\mathcal{A}^0_\mathcal{E}$. Fix an irreducible component $D_\alpha$ of~$D$, a point $x\in D_\alpha \setminus \bigcup_{\beta\neq\alpha} D_\beta$ and an open neighborhood $U$ of $x$ with coordinates $(z_1,z_2)$ such that $U\cap D_\alpha = \{z_1=0\}$. The matrix of $\Theta$ from \eqref{parametrizedbranched} in $\big(\der{}{z_1},\der{}{z_2}\big)$ is the difference between the matrices of $\nabla'$ and $\nabla$. It has the form \eqref{formnottau} by Proposition~\ref{necessarynottau}. By Lemma~\ref{caracterisationintermediate}, and since $\nabla$ is already an element of $\mathcal{A}^0_{\mathcal{E},1}$ we get that $\nabla'$ is again an element of $\mathcal{A}^0_{\mathcal{E},1}$, that is $(b)$ is true.
\end{proof}

In particular, in order to prove that a complex manifold $M$ admits a spiral holomorphic branched holomorphic affine connection with poles at $D$, and with associated submodule $\mathcal{E}$, it is sufficient to prove that there exists a holomorphic connection $\nabla$ on $\mathcal{E}$ which does not belong to~$\mathcal{A}^0_{\mathcal{E},1}$, i.e., such that for any irreducible component $D_\alpha$ of $D$, there exists $x\in D_\alpha$ such that no geodesic $\gamma$ of $\nabla$ at $x$ is contained in $D_\alpha$.

\subsection{Example in any dimension}
We now construct an example, for any $n\geq 1$, of a compact complex manifold $M$ of dimension~$n$, equipped with a submodule $\mathcal{E}\subset TM(* D)$ with $TM\subset \mathcal{E}$ and an object $\nabla \in \mathcal{A}^0_{\mathcal{E},\tau}$. Namely, $M$ are Hopf manifolds and $\nabla$ is constructed from branched coverings between these manifolds and from the holomorphic affine structure coming from their universal covering.

Pick any $\lambda \in [0,1]$, and define $\Gamma'$ and $\Gamma$ to be respectively the abelian groups spanned by the linear automorphisms $\lambda^2 {\rm Id}_{\mathbb{C}^n}$ and $(z_1,z_2,\dots,z_n)\mapsto \big(\lambda z_1, \lambda^2 z_2,\dots,\lambda^2 z_n\big)$ of $\mathbb{C}^n$. Let $\tilde{M}=\mathbb{C}^n \setminus \{0\}$, and let $M=\Gamma \backslash \tilde{M}$ and $M' = \Gamma' \backslash \tilde{M}$. $M$ and $M'$ are Hopf manifolds associated with~$\Gamma$ and $\Gamma'$, so these are complex compact manifolds. Since $\Gamma'$ is a subgroup of the affine group of~$\mathbb{C}^n$, the canonical affine structure $\nabla_0$ of $\mathbb{C}^n$ \big(restricted to $\tilde{M}$\big) descends as a holomorphic affine connection $\nabla'$ on $M'$.

The map $\tilde{f}\colon \tilde{M} \longrightarrow \tilde{M}$ given by $\tilde{f}(z_1,\dots,z_n) =\big(z_1^2,\dots,z_n\big)$ is equivariant for the actions of~$\Gamma$ and $\Gamma'$. Thus we obtain a map $f\colon M\longrightarrow M'$ defined by the commutative diagram\vspace{-1.5mm}
 \begin{gather*}
\xymatrix{ \tilde{M} \ar[r]^{\tilde{f}} \ar[d] & \tilde{M}\ar[d] \\ M \ar[r]_{f} & M'. }\vspace{-1.5mm}
\end{gather*}
By construction, $f$ is a double covering with ramification locus the divisor $D$ of $M$ obtained as the quotient of $\tilde{D}=\{z_1=0\}$ by the action of $\Gamma$. The pullback $\nabla=f^\star \nabla'$ is an object of~$\mathcal{A}^0_\mathcal{E,\tau}$ where $\mathcal{E}=\mathcal{O}_M\otimes_{f^{-1}\mathcal{O}_{M'}} f^{-1}TM'$. Indeed, by construction, $\nabla$ pulls back to $\mathbb{C}^n$ (first to $\tilde{M}$, then to $\mathbb{C}^n$ by the Hartog's extension theorem) as the $\Gamma$-invariant meromorphic affine connection $\tilde{\nabla} = \tilde{f}^\star \nabla_0$. Now, the curve $\tilde{\Sigma} =\{z_2 = \dots = z_n = 0\}$ projects onto itself through $\tilde{f}$, and is a geodesic for $\nabla_0$. Thus, its projection $\Sigma$ on $M$ is a geodesic for $\nabla$, and by construction~$\Sigma$ intersects $D$ exactly at one point, namely the class of $(0,\dots,0)$. Note that these examples are flat, in the sense that the corresponding meromorphic Cartan geometries are flat.
However, we mention the existence of non-flat examples due to the following construction, by S.~Cantat in \cite[Example 5.3]{Cantat} of a complex compact manifold $M$, equipped with a holomorphic parallelism with non-trivial structure constants, and with a multiple branched cover $f\colon M \longrightarrow M$ as a self-map.

 Let $H_3$ be the complex Heisenberg group, that is the subgroup of ${\rm SL}_3(\mathbb{C})$ of upper-triangular matrices with ones on the diagonal. The subgroup $\Gamma:=H_3(\mathbb{Z}[i])$ of elements with coefficients in~$\mathbb{Z}[i]$ is a cocompact lattice of $H_3$. The quotient $M=H_3/\Gamma$ is then a complex compact manifold. It is equipped with the holomorphic parallelism $\big(\overline{Z}_1,\overline{Z}_2,\overline{Z}_3\big)$ obtained from any basis of right-invariant holomorphic vector fields on $H_3(\mathbb{C})$. We denote by $\nabla'$ the unique holomorphic affine connection on $M'=M$ such that $\nabla'\big(\overline{Z}_i\big)=0$. By construction, the torsion of $\nabla'$ is not zero. It is proved in \cite[Example 5.3]{Cantat} that there exists a finite surjective morphism $f\colon M \longrightarrow M'$ with ramification locus a non-trivial divisor $D$ of $M$. We can reproduce the constructions used with the Hopf manifolds to obtain an object $\nabla=f^\star \nabla'$ of $\mathcal{A}^0_{\mathcal{E}}$ where $\mathcal{E}= \mathcal{O}_M\otimes_{f^{-1}\mathcal{O}_{M'}} f^{-1} TM'$. Then we use the following lemma.

 \begin{Lemma} Let $f\colon M\longrightarrow M'$ be a branched cover between two complex manifolds $M$, $M'$, with ramification locus $D\subset M$ $($we denote $D'=f_*(D))$, and $\nabla' \in \mathcal{A}^0_{\mathcal{E}',\tau}$ where $\mathcal{E}'$ is any submodule of $TM'(* D')$ as in Definition~{\rm \ref{branchedaffine}}. Then the pullback $\nabla=f^\star \nabla'$ is an object of $\mathcal{A}^0_{\mathcal{E},\tau}$ where~$\mathcal{E}=\mathcal{O}_M \otimes_{f^{-1}\mathcal{O}_{M'}} f^{-1} \mathcal{E}'$. \end{Lemma}

  \begin{proof} The fact that $\nabla$ is an object of $\mathcal{A}^0_{\mathcal{E}}$ is clear from the equivalence Corollary~\ref{equivalencetorsion} and the definition of a holomorphic branched Cartan geometry. Now, let $D_\alpha$ be any irreducible component of $D$, $D'_\alpha$ its projection through $f$. Since $\nabla'$ is a spiral branched holomorphic affine connection, by Lemma~\ref{lemmabranchedaffine} there exist $x'_0 \in D'_\alpha$ and a geodesic $\Sigma'$ for $\nabla'$ such that $\Sigma'\cap D'_\alpha = \{x'_0\}$. Let $x_0$ be any point of the fiber $f^{-1}(x'_0) \subset D_\alpha$, and $\Sigma=f^{-1}(\Sigma')$. Using the characterization of Lemma~\ref{caracterisationbranchedgeodesics} and the definition of the pullback, we obtain that $\Sigma$ is a geodesic of $\nabla$. By construction, $\Sigma \cap D_\alpha$ is the finite set of points $f^{-1}(x'_0)$, so we can consider a neighborhood $U$ of~$x_0$ such that $\Sigma\cap U$ is a geodesic of $\nabla$, intersecting $D_\alpha$ exactly at $x_0$. So $\nabla \in \mathcal{A}^0_{\mathcal{E},\tau}$.
\end{proof}

\vspace{-1mm}

\subsection*{Acknowledgements}
We are very grateful to the referees for their helpful work and their constructive comments.
We would like also to thank Sorin Dumitrescu for introducing the author to the Cartan geometries. We finally thank the Laboratoire de Math\'ematiques Blaise Pascal, in Clermont-Ferrand, for its hospitality.

\pdfbookmark[1]{References}{ref}
\LastPageEnding


\begin{thebibliography}{99}
\footnotesize\itemsep=0pt

\bibitem{Atiyah}
Atiyah M.F., Complex analytic connections in fibre bundles, \href{https://doi.org/10.2307/1992969}{\textit{Trans.
 Amer. Math. Soc.}} \textbf{85} (1957), 181--207.

\bibitem{Lunts}
Bernstein J., Lunts V., Equivariant sheaves and functors, \textit{Lecture Notes
 in Math.}, Vol.~1578, \href{https://doi.org/10.1007/BFb0073549}{Springer}, Berlin, 1994.

\bibitem{BD}
Biswas I., Dumitrescu S., Branched holomorphic {C}artan geometries and
 {C}alabi--{Y}au manifolds, \href{https://doi.org/10.1093/imrn/rny003}{\textit{Int. Math. Res. Not.}} \textbf{2019}
 (2019), 7428--7458, \href{https://arxiv.org/abs/1706.04407}{arXiv:1706.04407}.

\bibitem{BDM}
Biswas I., Dumitrescu S., McKay B., Cartan geometries on complex manifolds of
 algebraic dimension zero, \href{https://doi.org/10.46298/epiga.2019.volume3.4460}{\textit{\'Epijournal G\'eom. Alg\'ebrique}}
 \textbf{3} (2019), 19, 10~pages, \href{https://arxiv.org/abs/1804.08949}{arXiv:1804.08949}.

\bibitem{BlazquezCasale}
Bl\'azquez-Sanz D., Casale G., Parallelisms \& {L}ie connections,
  \href{https://doi.org/10.3842/SIGMA.2017.086}{\textit{SIGMA}} \textbf{13} (2017), 086, 28~pages, \href{https://arxiv.org/abs/1603.07915}{arXiv:1603.07915}.

\bibitem{Brunella}
Brunella M., Birational geometry of foliations, \textit{IMPA Monographs},
 Vol.~1, \href{https://doi.org/10.1007/978-3-319-14310-1}{Springer}, Cham, 2015.

\bibitem{Cantat}
Cantat S., Endomorphismes des vari\'et\'es homog\`enes, \href{https://doi.org/10.1007/978-3-319-14310-1}{\textit{Enseign. Math.}}
 \textbf{49} (2003), 237--262.

\bibitem{Cap4}
\v{C}ap A., Infinitesimal automorphisms and deformations of parabolic
 geometries, \href{https://doi.org/10.4171/JEMS/116}{\textit{J.~Eur. Math. Soc.}} \textbf{10} (2008), 415--437,
 \href{https://arxiv.org/abs/math.DG/0508535}{arXiv:math.DG/0508535}.

\bibitem{Cap}
\v{C}ap A., Slov\'ak J., Parabolic geometries. {I}: Background and general
 theory, \textit{Math. Surveys Monogr.}, Vol.~154, \href{https://doi.org/10.1090/surv/154}{American Mathematical Society}, Providence, RI, 2009.

\bibitem{Cartan2}
Cartan E., Les groupes d'holonomie des espaces g\'en\'eralis\'es,  \href{https://doi.org/10.1007/BF02629755}{\textit{Acta
 Math.}} \textbf{48} (1926), 1--42.

\bibitem{Chevalley}
Chevalley C., Sur certains groupes simples, \href{https://doi.org/10.2748/tmj/1178245104}{\textit{Tohoku Math.~J.}} \textbf{7}
 (1955), 14--66.

\bibitem{Deligne}
Deligne P., \'Equations diff\'erentielles \`a points singuliers r\'eguliers,
 \textit{Lecture Notes in Math.}, Vol.~161, \href{https://doi.org/10.1007/BFb0061194}{Springer}, Berlin, 1970.

\bibitem{DumitrescuMetriques}
Dumitrescu S., M\'etriques riemanniennes holomorphes en petite dimension,
 \href{https://doi.org/10.5802/aif.1870}{\textit{Ann. Inst. Fourier (Grenoble)}} \textbf{51} (2001), 1663--1690.

\bibitem{Dumitrescu1}
Dumitrescu S., Killing fields of holomorphic {C}artan geometries,
 \href{https://doi.org/10.1007/s00605-009-0135-x}{\textit{Monatsh. Math.}} \textbf{161} (2010), 145--154, \href{https://arxiv.org/abs/0902.2193}{arXiv:0902.2193}.

\bibitem{Dumitrescu0}
Dumitrescu S., Meromorphic almost rigid geometric structures, in Geometry,
 Rigidity, and Group Actions, \textit{Chicago Lectures in Math.}, University Chicago
 Press, Chicago, IL, 2011, 32--58, \href{https://arxiv.org/abs/0805.4506}{arXiv:0805.4506}.

\bibitem{DumitrescuQuasi}
Dumitrescu S., An invitation to quasihomogeneous rigid geometric structures, in
 Bridging Algebra, Geometry, and Topology, \textit{Springer Proc. Math.
 Stat.}, Vol.~96, \href{https://doi.org/10.1007/978-3-319-09186-0_7}{Springer}, Cham, 2014, 107--123.

\bibitem{Erhesmanh}
Ehresmann C., Les connexions infinit\'esimales dans un espace fibr\'e
 diff\'erentiable, in S\'eminaire {B}ourbaki, {V}ol.~1, Soc. Math. France,
 Paris, 1995, Exp. No.~24, 153--168.

\bibitem{Gauduchon}
Gauduchon P., La {$1$}-forme de torsion d'une vari\'et\'e hermitienne compacte,
 \href{https://doi.org/10.1007/BF01455968}{\textit{Math. Ann.}} \textbf{267} (1984), 495--518.

\bibitem{GrauertPeternellRemmert}
Grauert H., Peternell T., Remmert R., Several complex variables. {VII}.
 Sheaf-theoretical methods in complex analysis, \textit{Encycl. Math. Sci.}, Vol.~74, \href{https://doi.org/10.1007/978-3-662-09873-8}{Springer}, Berlin, 1994.

\bibitem{Gromov}
Gromov M., Rigid transformations groups, in G\'eom\'etrie diff\'erentielle
 ({P}aris, 1986), \textit{Travaux en Cours}, Vol.~33, Hermann, Paris, 1988,
 65--139.

\bibitem{Hotta}
Hotta R., Takeuchi K., Tanisaki T., {$D$}-modules, perverse sheaves, and
 representation theory, \textit{Progr. Math.}, Vol. 236, \href{https://doi.org/10.1007/978-0-8176-4523-6}{Birkh\"auser}, Boston,
 MA, 2008.

\bibitem{InoueKobayashi}
Inoue M., Kobayashi S., Ochiai T., Holomorphic affine connections on compact
 complex surfaces, \textit{J.~Fac. Sci. Univ. Tokyo Sect. IA Math.}
 \textbf{27} (1980), 247--264.

\bibitem{LeBrun}
LeBrun C., Spaces of complex null geodesics in complex-{R}iemannian geometry,
 \href{https://doi.org/10.2307/1999312}{\textit{Trans. Amer. Math. Soc.}} \textbf{278} (1983), 209--231.

\bibitem{McKay}
McKay B., Characteristic forms of complex {C}artan geometries, \href{https://doi.org/10.1515/ADVGEOM.2010.044}{\textit{Adv.
 Geom.}} \textbf{11} (2011), 139--168, \href{https://arxiv.org/abs/0704.2555}{arXiv:0704.2555}.

\bibitem{Nomizu}
Nomizu K., On local and global existence of {K}illing vector fields,
\href{https://doi.org/10.2307/1970148}{\textit{Ann. of Math.}} \textbf{72} (1960), 105--120.

\bibitem{Sabbah}
Sabbah C., Isomonodromic deformations and {F}robenius manifolds. An
 introduction, \textit{Universitext}, \href{https://doi.org/10.1007/978-1-84800-054-4}{Springer}, London, 2007.

\bibitem{Sharpe}
Sharpe R.W., Differential geometry. Cartan's generalization of Klein's Erlangen
 program, \textit{Grad. Texts in Math.}, Vol.~166, Springer, New York, 1997.

\bibitem{Wolf}
Wolf J.A., Spaces of constant curvature, 6th ed., \href{https://doi.org/10.1090/chel/372}{AMS Chelsea Publishing},
 Providence, RI, 2011.

\end{thebibliography}
\end{document}